\documentclass[11pt]{amsart}       

\usepackage{amscd,amsmath,amssymb,amsfonts,xspace,mathrsfs}
\usepackage{graphicx}
\usepackage{color}

\numberwithin{equation}{section}

 \newtheorem{theorem}{Theorem}
\numberwithin{theorem}{section}
  \newtheorem{corollary}[theorem]{Corollary}
   \newtheorem{proposition}[theorem]{Proposition}
\theoremstyle{definition}
 \newtheorem{definition}[theorem]{Definition}

\theoremstyle{remark}
 \newtheorem{remark}[theorem]{Remark}

\def\bea{\begin{eqnarray}}
\def\eea{\end{eqnarray}}
\def\be{\begin{equation}}
\def\ee{\end{equation}}
\def\ba{\begin{align}}
\def\ea{\end{align}}
\def\bse{\begin{subequations}}
\def\ese{\end{subequations}}

\def\({\left(}
\def\){\right)}
\def\[{\left[}
\def\]{\right]}
\def\<{\left\langle}
\def\>{\right\rangle}

\newcommand{\IR}{\mathbb{R}}
\newcommand{\IC}{\mathbb{C}}
\newcommand{\IZ}{\mathbb{Z}}

\newcommand{\cB}{\mathcal{B}}

\newcommand{\cD}{\mathcal{D}}

\newcommand{\cI}{\mathcal{I}}

\newcommand{\cM}{\mathcal{M}}

\newcommand{\cP}{\mathcal{P}}
\newcommand{\cQ}{\mathcal{Q}}

\newcommand{\cZ}{\mathcal{Z}}

\newcommand{\I}{\mathrm{i}}
\newcommand{\de}{\mathrm{d}}

\DeclareMathOperator{\sign}{sign}

\DeclareMathOperator{\Erf}{Erf}
\DeclareMathOperator{\Erfc}{Erfc}
\DeclareMathOperator{\Arctan}{Arctan}
\DeclareMathOperator{\Pv}{Pv}

\renewcommand{\Im}{\operatorname{Im}}

\def\cl0{\tilde c_0}
\newcommand{\expe}[1]{{\bf e}\!\left( #1\right)}
\newcommand{\hf}{\frac{1}{2}}
\def\haf{\textstyle{1\over 2}}
\newcommand{\p}{\partial}
\newcommand{\pa}{\partial}
\newcommand{\nn}{\nonumber}

\newcommand{\bfb}{{ b}}
\newcommand{\bfc}{{ c}}

\newcommand{\bfk}{{ k}}

\newcommand{\bfp}{{ p}}

\newcommand{\bfx}{{ x}}

\newcommand{\bfmu}{{ \mu}}
\newcommand{\bfnu}{{ \nu}}

\newcommand{\non}{\nonumber}

\def\CYm{{{\mathfrak{Y}}}}

\newcommand{\Cn}{C} 
\newcommand{\Cudp}{C_3^\star}
\newcommand{\Cdtp}{C_1^\star}
\newcommand{\Ctup}{C_2^\star}

\def\bp#1{#1}

\begin{document}

\title[Indefinite theta series]{Indefinite theta series and generalized error functions\footnote{
Preprint: L2C:16-078, IPHT-T16/058, TCDMATH 16-09, CERN-TH-2016-142, arXiv:1606.05495}}

\author{Sergei Alexandrov}
\address{Laboratoire Charles Coulomb (L2C), UMR 5221 CNRS-Universit\'e de
Montpellier, F-34095, Montpellier, France}
\email{salexand@univ-montp2.fr}

\author{Sibasish Banerjee}
\address{IPhT, CEA, Saclay, Gif-sur-Yvette, F-91191, France}
\email{sibasish.banerjee@cea.fr}

\author{Jan Manschot}
\address{School of Mathematics, Trinity College, Dublin 2, Ireland}
\email{manschot@maths.tcd.ie}

\author{Boris Pioline}
\address{CERN, TH Department, Case C01600, CERN, CH-1211 Geneva 23, Switzerland}
\address{Sorbonne Universit\'e, Campus Pierre et Marie Curie, UMR 7589, LPTHE, F-75005, Paris, France}
\address{CNRS, UMR 7589, LPTHE, 4 place Jussieu, F-75005, Paris, France}

\email{pioline@lpthe.jussieu.fr}

\begin{abstract}
Theta series for lattices with indefinite signature $(n_+,n_-)$ arise in many areas of mathematics
including representation theory and enumerative algebraic geometry. Their
modular properties are well understood in the Lorentzian case ($n_+=1$),
but have remained obscure when $n_+\geq 2$. Using a higher-dimensional generalization
of the usual (complementary) error function, discovered in an independent physics project,
we construct the modular completion of a class of `conformal' holomorphic theta series ($n_+=2$).
As an application, we determine the modular properties of a generalized Appell-Lerch sum attached
to the lattice ${\operatorname A}_2$, which arose in the study of rank 3 vector bundles on $\mathbb{P}^2$.
The extension of our method to $n_+>2$ is outlined.
\end{abstract}

\maketitle


\section{Introduction}\label{sec:intro}

Theta series appear in a variety of subjects in mathematics and physics and
provide a large class of functions exhibiting modular properties.
Theta series for lattices with \bp{negative} definite
signature\footnote{\bp{Throughout this work we use admittedly unusual
    sign conventions in which a negative definite quadratic form leads
    to a convergent holomorphic theta series. Correspondingly, a
    time-like vector has positive norm whereas a space-like vector has
    negative norm.}} are well-known examples of holomorphic modular forms.
While holomorphic theta series for lattices with indefinite signature have been
studied since Hecke \cite{zbMATH02591849}, their modular properties
 are \bp{not well understood in general}. 
 Motivated by Donaldson invariants of four-manifolds, G\"ottsche and Zagier \cite{MR1623706} studied holomorphic theta series obtained by summing over \bp{cones in} 
 lattices with Lorentzian signature $(1,n-1)$. 
Zwegers \cite{Zwegers-thesis} succeeded in determining their modular properties
by constructing a non-holomorphic modular completion, \bp{laying down the basis for} the modern understanding of Ramanujan's mock theta
functions \cite{MR2605321, Zwegers-thesis}. \bp{An alternative route has been followed by Kudla
and Millson \cite{KudlaMillson},  who constructed cohomological indefinite theta
 series for any signature $(n_+,n_-)$, which are holomorphic in  cohomology.}

An important open question is to understand the modularity of \bp{scalar-valued
indefinite theta series for
lattices in general signature.}  In this work we investigate  a general class of convergent 
theta series for lattices with signature $(2,n-2)$. Since $O(2,n-2)$ is
the group of conformal symmetries of $\IR^{1,n-3}$, we refer to these objects as
`conformal theta series'. Holomorphic conformal theta series are obtained by
introducing an appropriate (locally constant) kernel $\Phi_2(x)$ in the sum, which restricts it to a subset of the lattice where the quadratic form is \bp{negative}.
We obtain the modular completion  by replacing $\Phi_2(x)$ with
a smooth kernel $\widehat \Phi_2(x)$ which asymptotes to $\Phi_2(x)$
exponentially fast in the limit $|x|\to \infty$ and satisfies the assumptions of Vign\'eras' theorem \cite{Vigneras:1977}.
The kernel $\widehat \Phi_2(x)$ is expressed in terms of special functions
$E_2$ and $M_2$ generalizing the usual (complementary) error
function. For special choices of parameters, $\widehat \Phi_2(x)$ may factorize into
a product of two kernels $\widehat \Phi_1$ for Lorentzian lattices. For such degenerate cases recent results were obtained in
\cite{BRZ2015, Mortenson:2016, WRaum:2012}. The
functions $E_2$ and $M_2$ were discovered in our study of instanton effects in string theory using twistor techniques
\cite{abmp-to-appear}. No familiarity with these subjects is however required for the present
paper. While we focus on the conformal case for most of this work, in the last section we
introduce higher dimensional versions of $E_2$ and $M_2$, which can be used
to determine the modular completion  of indefinite theta series for lattices
with arbitrary signature $(n_+,n_-)$.

This progress on  indefinite theta series allows us to address the modularity of another
important class of $q$-series with a wide range of applications in
mathematics and mathematical physics,  namely Appell-Lerch sums. The
classical Appell-Lerch sum introduced in \cite{Appell:1886, Lerch:1892}
is well-known to be related to a Lorentzian theta series, which allows
to deduce its modular properties \cite{Zwegers-thesis}. Generalized
Appell-Lerch sums involving denominators of higher degree have appeared in the
context of topological invariants of moduli spaces of vector bundles on
4-manifolds \cite{Manschot:2014cca}. Upon expanding the factors in the denominator into
geometric series, they can be recast as indefinite theta series for lattices with signature
$(n_+,n_-)$. Our techniques can be used to characterize the modular
properties of such generalized Appell-Lerch sums. As a demonstration of the power of our
approach, we determine the modular completion of the generalized Appell-Lerch
sum attached to the lattice ${\operatorname A}_2$, which arises in the study of rank 3 vector bundles
on $\mathbb{P}^2$.

In the rest of this introduction, we review Appell-Lerch sums and their generalization.
Moreover, we give a brief summary for the interested reader
on the appearance of indefinite theta series in the context of gauge theory and vector bundles,
and on the physical background which led to the new techniques described in this work.
(The hurried reader can however safely skip to Section \ref{sec_vz}.)

\subsection{Indefinite theta series and Appell-Lerch sums}

As mentioned above, a class of functions which are closely related to theta series for Lorentzian lattices
are Appell-Lerch sums \cite{Appell:1886, Lerch:1892}.
The classical Appell-Lerch sum is a \bp{function of $\tau\in\mathbb{H}$ in the Poincar\'e upper half-plane  and 
$u,v\in \mathbb{C}\backslash(\tau \mathbb{Z}+\mathbb{Z})$, holomorphic in $\tau$ and meromorphic in $u$ and $v$, defined by}
\be
\label{ALsum}
\mu(u,v,\tau) :=
\frac{e^{\pi\I u}}{\theta_{1}(v,\tau)}\sum_{n\in \mathbb{Z}} \frac{(-1)^nq^{n(n+1)/2}e^{2\pi\I v n}}{1-e^{2\pi\I u}q^n},
\ee
where \bp{$q:=e^{2\pi\I\tau}$} and $\theta_1(v,\tau)$ is the Jacobi theta function
\be
\theta_1(v,\tau):=\I\sum_{r\in \mathbb{Z}+\frac{1}{2}}(-1)^{r-\frac{1}{2}}q^{\frac{r^2}{2}}e^{2\pi \I vr}.
\ee
This function plays a central role in the theory of $q$-series, and
in particular in  recent advances on mock modular forms
\cite{Zwegers-thesis}. Furthermore, the coefficients of $\mu(u,v)$ are
related to various enumeration problems. For example, specializations
of $\mu(u,v)$ are known to enumerate partition statistics (in
particular the rank of the partition \cite{MR2605321}) and
class numbers, i.e. the numbers of equivalence classes of binary
quadratic forms with fixed determinant \cite{bringmann:2007,
  Bringmann:2010sd, Zagier:1975}. The classical Appell-Lerch sum also appears in various other
areas of mathematics and physics, including the representation theory of superconformal algebras
 \cite{Eguchi:1988af, Kac:2001}, mirror symmetry for elliptic
curves \cite{Polishchuk:2001} and black hole counting in string theory
\cite{Dabholkar:2012nd}. It also shows up in the context of
gauge theories as reviewed in the next subsection.

The form of $\mu(u,v)$ naturally lends itself to higher-dimensional generalizations.
To state the most general form, we consider a \bp{negative} definite lattice of dimension $n_-$
with associated quadratic form $\mathcal{Q}$ and bilinear form $\mathcal{B}$.
Let $\{m_i\}$, $i=1,\dots,n_+$, be a set of $n_+$ vectors in $\mathbb{Z}^{n_-}$;
let $u\in\mathbb{C}^{n_+}\backslash
(\tau \mathbb{Z}^{n_+}+\mathbb{Z}^{n_+})$ and $v\in \mathbb{C}^{n_-}\backslash ({\rm zeroes\,\,of\,\,}\Theta_{\mathcal{Q}}(v,\tau))$.
Then we define the generalized Appell-Lerch sum by
\be
\label{genAppell}
\mu_{\mathcal{Q},\{m_i\}}(u,v,\tau):=\frac{e^{\pi\I \sum_{j=1}^{n_+}u_j}}{\Theta_{\mathcal{Q}}(v,\tau)}
\sum_{k\in \mathbb{Z}^{n_-}}\frac{q^{\bp{-} \tfrac12\mathcal{Q}(k)}e^{2\pi\I \mathcal{B}(v,k)}}
{\prod_{j=1}^{n_+}(1-e^{2\pi\I u_j} q^{\bp{-} \mathcal{B}(m_j,k)})},
\ee
with
\be
\label{Theta}
\Theta_{\cQ}(v,\tau):=\sum_{k\in \mathbb{Z}^{n_-}} q^{\bp{-}\tfrac12\cQ(k)} e^{2\pi \I \cB(v,k)}.
\ee
We refer to the pair $(n_+,n_-)$ as the ``signature'' of the Appell-Lerch sum,
since it is the signature of the quadratic form  obtained after expanding the $n_+$ factors
in the denominator of \eqref{genAppell} into  geometric sums.

When $n_+=1$, these functions reduce to the multi-variable Appell-Lerch sums studied in
\cite{Zwegers:2010}. For $n_+>1$, functions similar to $\mu_{\mathcal{Q},\{m_i\}}(u,v,\tau)$
appear as characters of affine superalgebras \cite[Eq. (0.13)]{Kac:2013}
and as generating functions of Gromov-Witten invariants of elliptic orbifolds \cite{BRZ2015}.
They also occur in the study of vector bundles on rational surfaces
\cite[Section 5]{Manschot:2014cca}.
In the latter case, the vectors $\{m_i\}$ are such that $\mu_{\mathcal{Q},\{m_i\}}(u,v)$
does not reduce to products of classical Appell-Lerch sums.

The classical Appell-Lerch sum can be viewed as an indefinite theta function for
a signature $(1,1)$ lattice as in Section \ref{secLorentzian}, with one of the reference vectors $C'$
taken to be null. Its modular completion $\widehat \mu(u,v)$ then
follows from Zwegers' results on indefinite theta functions \cite{Zwegers-thesis}.
The techniques discussed in Sections \ref{derror} and \ref{2n-2} also allow us to
determine the modular properties of the generalized  Appell-Lerch sums \eqref{genAppell}
for $n_+=2$. In Section \ref{genappell} we illustrate this for a specific example where
$\mathcal{Q}$ corresponds to the quadratic form of the ${\rm A}_2$ root lattice.
Following the procedure outlined in Section \ref{sec_gen} one may also determine the modular completion for $n_+>2$.

\subsection{Indefinite theta series, vector bundles and gauge theory}

An important motivation and source of insight for the theory of indefinite
theta series is the context of vector bundles on 4-manifolds
\cite{Gottsche:1999ix, MR1623706, Manschot:2014cca}, which we now
review briefly. Given a polarization $J$ of a smooth algebraic surface $S$, one considers
the moduli space $\mathcal{M}_J(\gamma)$ of vector bundles with fixed Chern
character $\gamma=(r,c_1, c_2)$ which are semi-stable
with respect to $J$. In favorable
circumstances, these moduli spaces are smooth and projective, and
one can consider their topological invariants such as the Poincar\'e
polynomial of $\mathcal{M}_J(\gamma)$ or the Donaldson invariant.

To explain the appearance of indefinite theta series in this context,
let $\Omega(\gamma,w;J)$ be the following variant of
the Poincar\'e polynomial of $\mathcal{M}_J(\gamma)$:
\be
\Omega(\gamma,w;J) :=w^{-d(\gamma)/2}\sum_{j=0}^{d(\gamma)}
b_{j}\bp{(\mathcal{M}_J(\gamma))}\, w^{j},
\ee
where  $b_i(X)=\dim H^i(X)$ is the $i$'th Betti number of $X$, and $d(\gamma)=\dim(\mathcal{M}_J(\gamma))$.
It is natural to consider generating functions of the
$\Omega(\gamma, w;J)$'s, holding the rank $r$ and first Chern class
$c_1$ fixed, and summing over the second Chern class $c_2$,
\be
h_{r,c_1}(z,\tau;J):=\sum_{c_2\in{\mathbb{N}}} \Omega(\gamma,w;J)\,q^{c_2-\delta},
\ee
with $w=e^{2\pi\I z}$ \bp{and $\delta$ a suitable rational constant}. 
For $w=-1$, these generating functions are
also known as Vafa-Witten partition functions, since they
appear in the path integral of topologically twisted
$\mathcal{N}=4$ supersymmetric gauge theory with gauge group $U(r)$
\cite{Vafa:1994tf}. The path integral in general however involves  additional
non-holomorphic terms (except in special
cases, e.g. when $S$ is a
K3 surface). Invariance under electric-magnetic duality predicts
that the path integral transforms as a modular form, and thus in turn
that $h_{r,c_1}(z,\tau;J)$ possesses modular properties. However the
non-holomorphic terms are poorly understood and the precise
modular properties of $h_{r,c_1}(z,\tau;J)$ are in general obscure.

A class of algebraic surfaces for which $h_{r,c_1}(z,\tau;J)$ can be
determined explicitly are the rational surfaces whose second homology
group is a lattice $\Lambda$ of Lorentzian signature. For these surfaces in rank 2,
$h_{2,c_1}(z,\tau;J)$ was found to involve an indefinite theta function,
which sums over a subset of $\Lambda$ determined by $J$
\cite{Gottsche:1999ix, MR1623706}. For rank $r>2$,
$h_{r,c_1}(z,\tau;J)$ can be expressed in terms of indefinite theta
functions for lattices of signature $(n_+,n_-)$ with
$n_+\leq r-1$ \cite{Manschot:2010nc, Manschot:2014cca, Toda:2014}. Specializing to the projective plane $S=\mathbb{P}^2$ and $r=2$, the theta
series can be brought to the form of the classical Appell-Lerch sum
\cite{Bringmann:2010sd, Yoshioka:1994} using the blow-up formula. For $r>2$, $h_{r,c_1}(z,\tau;J)$
can be instead expressed in terms of generalized
Appell-Lerch sums (\ref{genAppell}) \cite{Manschot:2014cca}. The main building block for $r=3$,
is a generalized Appell-Lerch sum based on the ${\rm A}_2$ root lattice,
for which various identities were proven in \cite{BMR2015}.
We determine its modular completion in Section \ref{genappell}. The
application of this modular completion to the functions $h_{3,c_1}(z,\tau;J)$
will be discussed elsewhere \cite{U3toappear}.

\subsection{Physical background}

The main building blocks used to construct completions of theta series of signature $(2,n-2)$
are the `double error' functions $M_2$ and $E_2$. These functions were
found by studying the following physics problem:
compute D-instanton corrections to the hypermultiplet moduli space
$\cM_H$ in type II string theory compactified on a compact Calabi-Yau three-fold $\CYm$ (see e.g.
\cite{Alexandrov:2011va,Alexandrov:2013yva} for surveys of recent progress on this issue).
In the context of type IIB theory, a subset of these D-instantons consists of D3-branes wrapping
a divisor $\cD$, bound to an arbitrary number of D1-branes and D(-1) instantons.
The S-duality symmetry of type IIB string theory requires these instanton corrections to be modular invariant,
while supersymmetry demands that they should be encoded in terms of a modular completion of some holomorphic theta series
on the twistor space $\cZ$ associated to $\cM_H$. In the one-instanton approximation, it was shown
in \cite{Alexandrov:2012au} that these theta series are indefinite theta
series of Lorentzian signature $(1,b_2(\CYm)-1)$ \`a la Zwegers. In \cite{Alexandrov:2016tnf,abmp-to-appear},
we consider two-instanton contributions to the metric, where wall-crossing phenomena start playing
a r\^ole,
and find that double D3-instantons are described by the completion of holomorphic theta series of conformal signature
$(2,2b_2(\CYm)-2)$, of the type  analyzed in this paper. The special functions $M_2$ and $E_2$  originate as Penrose-type integrals
of holomorphic transition functions on $\cZ$. The present work, however, does not assume any familiarity with D-instantons or twistors.

\subsection{Outline}

In \S\ref{sec_vz}, we recall Vign\'eras' theorem, which is the main tool in our analysis, and
Zwegers' construction of holomorphic theta series for Lorentzian lattices. As a warm up for
the higher signature case, we provide suggestive integral representations of
the error functions $E_1$ and $M_1$ which appear prominently in Zwegers' work.
In \S\ref{sec_m2e2}, we introduce special
functions $M_2$ (a solution of Vign\'eras' equation on $\IR^2$ with discontinuities on real codimension one loci,
exponentially \bp{decreasing}  at infinity) and $E_2$  (a smooth solution of Vign\'eras' equation on $\IR^2$,
locally constant at infinity), and use them to construct solutions of Vign\'eras' equation on
$\IR^{2,n-2}$. In \S\ref{sec_confth}, we use these building blocks to construct holomorphic
theta series for lattices with signature $(2,n-2)$, parametrized by
two pairs of time-like, \bp{positive definite} vectors,
and find their modular completion.  In \S\ref{sec_appell}, we apply this technology to the generalized
Appell-Lerch sum which arises in the study of vector bundles with rank 3 on $\mathbb{P}^2$.
In \S\ref{sec_gen}, we outline the extension of our method to signature $(n_+,n_-)$ with $n_+>2$,
and as a first step, construct the triple error functions $E_3$ and $M_3$ relevant for
the case $n_+=3$.

\subsection{Acknowledgments} J.~M. thanks Kathrin Bringmann, Thomas Creutzig, Robert Osburn,
Larry Rolen, Martin Westerholt-Raum, Don Zagier and Sander Zwegers for discussions
about the generating functions derived in \cite{Manschot:2010nc, Manschot:2014cca}. 
B.~P. is grateful to Trinity College Dublin for hospitality during part of this work.
\vspace{.3cm}
\\
\bp{
\noindent {\it Historical note}: After posting the first version of this manuscript on arXiv, 
we were informed that similar results to ours for signature $(2,n-2)$ have been obtained independently 
by Zagier and Zwegers in unpublished work from 2003, and more recently by 
Martin Westerholt-Raum \cite{Raum:2016}. 
In upcoming work, Zagier and Zwegers plan to discuss the modularity of indefinite theta series for lattices with general signature. 
The extension of the functions $E_2(C_1,C_2;x)$ and $M_2(C_1,C_2;x)$ to the case $\Delta_{12}=0$
discussed around \eqref{E2Delta0}, and the discussion of the case $C_1=C_2$ in Section \ref{sec_sig21}, were added in the second release of this work, following a suggestion by Zagier. 
The relation of our conformal theta series to the cohomological theta series of \cite{KudlaMillson}
was clarified  by Kudla 
\cite{Kudla:2016} who also provided weaker convergence conditions for
 the theta series $\vartheta_{\bfmu}[\Phi_2,0]$ than those stated in Theorem \ref{convergence2n}. 
 The construction of generalized error functions to arbitrary
signature along the lines suggested in Section 6 was worked out by Nazaroglu \cite{Nazaroglu:2016}, and the
geometry underlying the corresponding theta series was spelled out 
by Funke and Kudla \cite{FunkeKudla:2016}. Several works since then
have applied this
construction to find the modular completion of various $q$-series.
}

\section{Vign\'eras' theorem and Lorentzian theta series \label{sec_vz}}

In this section, we recall Vign\'eras' theorem \cite{Vigneras:1977}, which provides a general framework
for indefinite theta series of arbitrary signature. We then review Zwegers' construction of holomorphic
theta series of signature $(1,n-1)$, and discuss some useful integral representations of the
error and complementary error functions $E_1(u)=\Erf(u\sqrt{\pi})$ and
$M_1(u)=-\sign(u)\, \Erfc(|u|\sqrt{\pi})$ which play a central role
in obtaining the modular completion of these theta series.

\subsection{Vign\'eras' theorem}

Let $\Lambda$ be an $n$-dimensional lattice equipped with a symmetric bilinear form
$B(x,y)$, where $x,y\in \Lambda \otimes \mathbb{R}$, such that its associated quadratic form $Q(x)=B(x,x)$
has signature $(n_+,n_-)$ and is integer valued, i.e. $Q(k)\in\IZ$ for $k\in\Lambda$.
Furthermore, let $\bfp\in\Lambda$  be a characteristic vector
(such that $Q(\bfk)+B(\bfk,\bfp)\in 2\mathbb{Z}$, $\forall \,\bfk \in \Lambda$),
$\bfmu\in\Lambda^*/\Lambda$ a glue vector, and $\lambda$ an arbitrary integer.
With the usual notations \bp{$\expe{x}:=e^{2\pi\I x}$,  
$q=\expe{\tau}$,} $\tau=\tau_1+\I\tau_2\in \mathbb{H}$ and
$b,c\in \Lambda\otimes \mathbb{R}$, we consider the following family of theta series
\be
\label{Vignerasth}
\vartheta_{\bfmu}[\Phi,\lambda](\tau, \bfb, \bfc):=\tau_2^{-\lambda/2} \sum_{{\bfk}\in \Lambda+\bfmu+\hf\bfp}
(-1)^{B(\bfk,\bfp)}\,\Phi(\sqrt{2\tau_2}(\bfk+\bfb))\, q^{-\frac12 Q(\bfk+\bfb)}\,\expe{B(\bfc,\bfk+\haf\bfb)}
\ee
defined by the kernel $\Phi(\bfx)$ \bp{on $\Lambda\otimes\mathbb{R}$}. 
The theta series \eqref{Vignerasth} is independent
of the choice of characteristic vector $p$ so we omit it in the notation.
We choose the kernel $\Phi(\bfx)$ so that
$f(\bfx):=\Phi(\bfx)\,e^{\frac{\pi}{2}Q(\bfx)}\in L^1(\Lambda\otimes\mathbb{R})$,
which ensures the absolute convergence of the sum.
For any such $\Phi$ and $\lambda\in\IZ$, \eqref{Vignerasth} satisfies the
following quasi-periodicity properties
\be
\begin{split}
\label{Vigell}
\vartheta_{\bfmu}[\Phi ,\lambda]\left(\tau, \bfb+\bfk,\bfc\right) =&\,(-1)^{B(\bfk,\bfp)}\,
\expe{-\haf B(\bfc,\bfk)} \vartheta_{\bfmu}[\Phi,\lambda]\left(\tau, \bfb,\bfc\right),
\\
\vphantom{A^A \over A_A}
\vartheta_{\bfmu}[\Phi, \lambda]\left(\tau, \bfb,\bfc+\bfk \right)=&\,(-1)^{B(\bfk,\bfp)}\,
\expe{\haf B(\bfb, \bfk)} \vartheta_{\bfmu}[\Phi ,\lambda]\left(\tau, \bfb,\bfc\right).
\end{split}
\ee

Now let us require that $\Phi,\lambda$ satisfy the following two conditions:
\begin{enumerate}
\item[i)]
Let $D(\bfx)$ be any differential operator of order $\leq 2$, and
$R(\bfx)$ any polynomial of degree $\leq 2$. Then \bp{$f(\bfx)$}
is such that $f(\bfx)$, $D(\bfx)f(\bfx)$ and $R(\bfx)f(\bfx)\in$
$L^2(\Lambda\otimes\mathbb{R})\bigcap L^1(\Lambda\otimes\mathbb{R})$.
\item[ii)]
$\Phi(\bfx)$ and $\lambda$ satisfy
\be
\label{Vigdif}
\left[ B^{-1}(\partial_{\bfx},\partial_{\bfx})+ 2\pi \bfx\pa_{\bfx}  \right] \Phi(\bfx)  = 2\pi \lambda\,  \Phi(\bfx),
\ee
where $B^{-1}$ is the bilinear form on the dual lattice $\Lambda^*$,
whose matrix representation is the inverse of the matrix representation of $B$,
\bp{ $B^{-1}(\partial_{\bfx},\partial_{\bfx})$ is the Laplace operator and 
and $x\partial_x$ is the Euler operator.}
\end{enumerate}

\begin{theorem}[\cite{Vigneras:1977}]
Under the conditions i) and ii), the theta series $\vartheta_{\bfmu}[\Phi,\lambda](\tau, \bfb, \bfc)$
transforms as a vector-valued Jacobi form of weight
\bp{$\lambda+\tfrac{n}{2}$}. Namely, it satisfies \eqref{Vigell} and 
\be
\begin{split}
\vartheta_{\bfmu}[\Phi,\lambda]\left( -1/\tau, \bfc,-\bfb\right)
=&\,\frac{(-\I\tau)^{\lambda+\frac{n}{2}}}{\sqrt{|\Lambda^*/\Lambda|}}\,
\expe{ \bp{\tfrac{Q(\bfp)+\lambda+n_+}{4}}} \sum_{\bfnu\in\Lambda^*/\Lambda}
\expe{B(\bfmu,\bfnu)}
\vartheta_{\bfnu}[\Phi,\lambda]\left(\tau, \bfb,\bfc\right),
\\
\vartheta_{\bfmu}[\Phi,\lambda] \left(\tau+1, \bfb,\bfc+\bfb\right)
=&\,\expe{-\tfrac12 Q(\bfmu+\tfrac12 \bfp)}
\vartheta_{\bfmu}[ \Phi,\lambda] \left( \tau, \bfb,\bfc\right).
\end{split}
\label{eq:thetatransforms}
\ee
\end{theorem}

\begin{remark} \label{rk_hol}  
The elliptic and modular transformations above are those of a theta series with
characteristics $b$ and $c\in\Lambda \otimes \mathbb{R}$.
They are  related to the standard multi-variable Jacobi forms in the sense of \cite{MR781735} under the change of variables
\be
\bp{\tilde {\vartheta}_{\mu}(\tau, z) := \expe{\tfrac12 B(\bfb,b\tau-c)}
\vartheta_{\bfmu}(\tau, \bfb,\bfc)} \,,
\qquad
z=b \tau-c\, .
\ee
Later in this work, remarks on the holomorphy of $\vartheta_{\bfmu}$ with respect to $\tau$ or
$z$ are to be understood as statements about $\tilde {\vartheta}_{\mu}\!\left(\tau, z\right)$.
\end{remark}

\begin{remark}
The class of theta series \eqref{Vignerasth} is closed under the action of
the Maass raising and lowering operators \bp{$\partial_{\tau}-\tfrac{\I w}{2\tau_2}$ and 
$\tau_2^2 \partial_{\bar\tau}$}, which map modular forms of weight $w$ to
modular forms of weight $w \bp{\pm} 2$:
\be
\label{covdertheta}
\begin{split}
\left(\partial_{\tau}-\tfrac{\I(\lambda+\tfrac{n}{2})}{2\tau_2}\right) \vartheta_{\bfmu}[\Phi,\lambda]
=&\,
\vartheta_{\bfmu}\! \left[-\tfrac{\I}{4} \left( \bfx \partial_{\bfx} \Phi
+ [\lambda+n+2\pi Q(\bfx)] \Phi \right), \lambda+2 \right],\\
\tau_2^2 \partial_{\bar\tau}\, \vartheta_{\bfmu}[\Phi,\lambda]
=&\,
\vartheta_{\bfmu}\! \left[\tfrac{\I}{4} \left( \bfx \partial_{\bfx} \Phi - \lambda \Phi\right) , \lambda-2 \right].
\end{split}
\ee
We refer to $\tau_2^2 \partial_{\bar\tau}$ and to its counterpart $\tfrac{\I}{4} ( \bfx \partial_{\bfx}  - \lambda)$,
as the `shadow' operators, and we occasionally omit the argument $\lambda$ when it is determined from $\Phi$ via \eqref{Vigdif}.
\end{remark}

\begin{remark}
For reasons which will become clear shortly, we shall be interested
in functions $\widehat\Phi(x)$ which asymptote  to
a  locally polynomial, homogeneous function $\Phi(x)$ of degree $\lambda$
with exponential accuracy as $|x|\to\infty$ along generic radial rays.
In this case, $\widehat\Phi(x)$ can be recovered from its shadow
$\Psi=\tfrac{\I}{4} ( \bfx \partial_{\bfx}  - \lambda)\widehat\Phi$ and from its value at infinity
$\Phi$ by integrating along radial rays,
\be
\label{PhifromPsi}
\widehat\Phi(x) = \Phi(x) +4\I \int_1^{\infty} \Psi(t x)\, \bp{\frac{\de t}{t^{\lambda+1}}\,.}
\ee
Inserting this in $\vartheta_{\bfp,\bfmu}[\widehat\Phi]$ and changing variable from $t$ to
$\bar w=\tau-2\I\tau_2 t^2$, we find that $\vartheta_{\bfp,\bfmu}[\widehat\Phi]$ and
$\vartheta_{\bfmu}[\Phi]$ differ by a term proportional to the Eichler integral of
$\vartheta_{\bfmu}[\Psi]$,
\be
\label{wPhiperiod}
\vartheta_{\bfmu}[\widehat\Phi,\lambda]\left(\tau,\bfb,\bfc\right)=
\vartheta_{\bfmu}[\Phi,\lambda]\left(\tau,\bfb,\bfc\right)
-4\int_{-\I\infty}^{\bar\tau} \vartheta_{\bfmu}
[\Psi ,\lambda-2]\left( \tau, \bar w, \bfb,\bfc\right) \bp{\frac{\de \bar w}{(\tau-\bar w)^2}} \,.
\ee
Here, for a real-analytic function  $\Psi(x)$ on $\Lambda\otimes \IR$,
$\vartheta_{\bfmu}[\Psi]\left(\tau, \bar w,\bfb,\bfc\right)$
is defined by analytically extending \eqref{Vignerasth} away from the locus $\bar\tau=\tau_1-\I \tau_2$, i.e
by replacing $\bar \tau$ with $\bar w$.
\end{remark}

\begin{remark} \label{rk_Borcherds}
The simplest application of this theorem is to choose an $n_+$-dimensional
time-like plane $\cP$ inside $\Lambda\otimes \IR$, and decompose $x=x_+ + x_-$ where $x_+\in\cP$ and $B(x_+,x_-)=0$.
The function
$\Phi(x)=e^{-\pi Q(x_+)}$ then satisfies the assumptions (i), (ii) with
$\lambda=-n_+$. The resulting  theta series $\vartheta_{\bfmu}[\Phi]$ is 
\bp{the familiar Siegel theta series, a
vector-valued Jacobi form
of weight $\tfrac{n_--n_+}{2}$ (recall our unusual sign convention for the quadratic form), 
which is however not holomorphic in $\tau$.} The more general Siegel theta series
depending on a homogenous polynomial of degree $(m_+,m_-)$ in $(x_+,x_-)$
constructed in \cite{0919.11036} can also be understood in this framework.
\end{remark}

\subsection{Lorentzian theta series}
\label{secLorentzian}
In order to obtain a theta series which is holomorphic in $\tau$, it is necessary
to choose a function $\Phi(x)$ which is locally homogeneous of degree $\lambda$. 
However, such functions
do not satisfy the conditions i) and ii) above (except if $\Phi$ is strictly constant,
which is only admissible if $B$ is negative definite), and the resulting theta series
will not be modular.  For Lorentzian signature $(n_+,n_-)=(1,n-1)$, the modular anomalies
of such indefinite theta series are now well
understood, thanks to the work of G\"ottsche and Zagier \cite{MR1623706} and Zwegers \cite{Zwegers-thesis}.
We recall the following theorem from their work:

\begin{theorem}\cite{MR1623706, MR2605321, Zwegers-thesis}
\label{modularLorentz}
Let $Q(x)$ be a quadratic form of signature $(1,n-1)$. For any vector $C$ with $Q(C)>0$, we
denote $E_1(C;x):=E_1\!\left( \tfrac{B(C,x)}{\sqrt{Q(C)}} \right)$ where
$E_1(u):=\Erf(u\sqrt{\pi})$ is the (rescaled) error function.
For any pair of  linearly independent vectors $C, C' \in \Lambda\otimes \IR$ such that $Q(C), Q(C'), B(C,C')>0$, the following holds:
\begin{itemize}
\item[i)]  The theta series $\vartheta_{\bfmu}[\Phi_1,0]$ with kernel
\be
\label{PhiZwegers0}
\Phi_1(x) := \frac12 \,\Bigl[ \sign B(C,x)  - \sign B(C',x) \Bigr]
\ee
is convergent, holomorphic in $\tau$  and $z$ (in the sense of Remark \ref{rk_hol}),
away from real codimension-1
loci where $B(C,k+b)=0$ or $B(C',k+b)=0$ for some $k\in\Lambda+\mu+\tfrac12 p$.
\item[ii)] The theta series
$\vartheta_{\bfmu}[\widehat\Phi_1,0]$ with kernel
\be
\label{PhiZwegers}
\widehat\Phi_1(x) := \frac12 \,\Bigl[ E_1(C;x)- E_1(C';x)\Bigr]
\ee
is a non-holomorphic vector-valued Jacobi form of weight \bp{$\tfrac{n}{2}$}.
\item[iii)] The shadow of $\vartheta_{\bfmu}[\widehat\Phi_1,0]$ is the Gaussian theta series
$\vartheta_{\bfmu}[\Psi_1,-2]$
with kernel
\be
\Psi_1(x) :=  \frac{\I}{4} \left[ \tfrac{B(C,x)}{\sqrt{Q(C)}}\, e^{-\frac{\pi B(C,x)^2}{Q(C)}}  -
\tfrac{B(C',x)}{\sqrt{Q(C')}}\, e^{-\frac{\pi B (C',x)^2}{Q(C')}}  \right].
\ee
\end{itemize}
\end{theorem}

\begin{proof}
To establish the convergence of  $\vartheta_{\bfmu}[\Phi_1,0]$, observe that
whenever $\{x,C,C'\}$ are linearly independent, they span a signature $(1,2)$ plane, so
the determinant $\Delta(x,C,C')$ of the Gram matrix of $B$ on $\{x,C,C'\}$
is strictly positive. The latter evaluates to
\be
\label{DxCC}
\Delta(x,C,C')= Q_{C,C'}(x)\, \Delta(C,C')
 -\left[ Q(C)\, B(C',x)^2 + Q(C')\, B(C,x)^2 \right],
\ee
where $\Delta(C,C')=Q(C) Q(C')- B(C,C')^2<0$ and
\be
Q_{C,C'}(x):=  Q(x)+\tfrac{2B(C,C')\, B(C,x) B(C',x)}{\Delta(C,C')}.
\ee
Thus $Q_{C,C'}(x)<0$. If $\{x,C,C'\}$ are linearly dependent,
$\Delta(x,C,C')$ vanishes, so $Q_{C,C'}(x)\leq 0$, with equality when
$B(C,x)=B(C',x)=0$, hence $x=0$.
Therefore, the quadratic form $Q_{C,C'}(x)$ is negative definite. Now, $\Phi_1(x)$ vanishes
unless $B(C,x)\, B(C',x)\leq 0$. For those $x$, since $B(C,C')>0$, we have $Q(x)\leq Q_{C,C'}(x)$, so
$f(x)=\Phi_1(x)\, e^{\tfrac{\pi}{2} Q(x)}$ is dominated by $e^{\tfrac{\pi}{2} Q_{C,C'}(x)}$. In particular,
it lies in $L^1(\Lambda\otimes \IR)$, so  $\vartheta_{\bfmu}[\Phi_1,0]$ converges.
Holomorphy in $\tau$ and $z$ follows from the fact that $\Phi_1(x)$ is locally constant.

To establish the convergence of $\vartheta_{\bfmu}[\widehat\Phi_1]$, we
decompose $E_1(u)=M_1(u)+\sign(u)$ where
\be
\label{defM1}
M_1(u):=-\sign(u)\, \Erfc(|u|\sqrt{\pi})
\ee
is exponentially \bp{decreasing}  at large $u$, obtaining
\be
\label{PhiZwegersdec}
\widehat\Phi_1(x) = \Phi_1(x) +
\frac12\, M_1(C;x)-\frac12\,M_1(C';x),
\ee
where we denote $M_1(C;x) = M_1\!\left( \tfrac{B(C,x)}{\sqrt{Q(C)}} \right)$.
Focusing on the second term in \eqref{PhiZwegersdec}, we observe that whenever $\{x,C\}$
are linearly independent, $\Delta(x,C):=  Q(x)Q(C)-B(C,x)^2< 0$,
so $Q_C(x):=  Q(x)-2\tfrac{B(C,x)^2}{Q(C)}$ $\leq Q(x)-\tfrac{B(C,x)^2}{Q(C)}<0$, while equality
holds only for $x=0$. Thus the quadratic form $Q_C(x)$ is negative definite.
Since $|M_1(u)|\leq e^{-\pi u^2}$ for all $u\in \IR$, it follows that
$f(x):=  M_1(C;x)\, e^{\tfrac{\pi}{2}Q(x)}$ is dominated by $e^{\tfrac{\pi}{2}Q_C(x)}$.
The same argument applies to the
last term in \eqref{PhiZwegersdec}. Thus, $\widehat\Phi_1(x)$ satisfies condition i) of Vign\'eras' theorem.
Moreover, using $E'_1(u)=2 e^{-\pi u^2}$
one can check that $\widehat \Phi_1$ satisfies \eqref{Vigdif} with $\lambda=0$. This proves that
$\vartheta_{\bfmu}[\widehat\Phi_1]$ is a vector-valued Jacobi form of weight \bp{$\tfrac{n}{2}$}.
Its shadow is easily computed using \eqref{covdertheta}.
\end{proof}

\begin{remark}
For the case at hand, Eq. \eqref{wPhiperiod} shows that
the non-holomorphic theta series $\vartheta_{\bfmu}[\widehat\Phi_1]$ decomposes into
the sum of the holomorphic theta series $\vartheta_{\bfmu}[\Phi_1]$ and an Eichler integral
of the Gaussian theta series $\vartheta_{\bfmu}[\Psi_1]$. Both terms transform anomalously
under modular transformations, but the anomalies cancel in the sum.
\end{remark}

\begin{remark}
\label{remarknullvectors}
When $C$ and $C'$ degenerate to null vectors in $\Lambda$, the shadow
$\Psi_1$ vanishes and $\vartheta_{\bfmu}[\Phi_1]$ becomes a meromorphic Jacobi form,
away from the loci where $B(C,k+b)=0$ or $B(C',k+b)=0$ for some $k\in\Lambda+\mu+\tfrac12 p$ \cite{MR1623706}.
\end{remark}

\subsection{Integral representations of $M_1$ and $E_1$}

Our aim in the remainder of this work will be to generalize this construction to
holomorphic theta series of signature $(n_+,n_-)$ with $n_+>1$.  To prepare the ground,
it will be useful to note that the complementary error function $M_1(u)$ defined in \eqref{defM1}
has a contour integral representation
\be
\label{defM1int}
M_1(u)= \frac{\I}{\pi} \int_{\ell}\,
e^{-\pi z^2 -2\pi \I z u}\,\frac{\de z}{z} ,
\ee
where the contour $\ell=\IR-\I u$ runs parallel to the real axis through the saddle point at $z=-\I u$.
Indeed, setting $z=u(v-\I)$, one finds
\be
\label{defM1int2}
M_1(u) = \frac{\I}{\pi}\, \sign(u) \int_{\IR} e^{-\pi u^2 (v^2+1)}\, \frac{\de v}{v-\I}
=-\frac{2}{\pi}\, \sign(u) \int_0^\infty e^{-\pi u^2(v^2+1)}\, \frac{\de v}{v^2+1}.
\ee
Using $\Erfc(z)=\tfrac{2}{\pi} \int_0^{\infty}e^{-v^2 (z^2+1)}\, \tfrac{\de v}{v^2+1}$
for $|{\rm Arg}(z)|\leq \tfrac{\pi}{4}$ \cite[7.7.1]{NIST:DLMF}, we recover
\eqref{defM1}.
The representations \eqref{defM1int} and \eqref{defM1int2} make several facts obvious.
First, from \eqref{defM1int2} $M_1(u)$ is a real-valued, odd function of $u$, $C^\infty$ away from $u=0$, and exponentially
\bp{decreasing}  as $|u|\to\infty$. Second, from the pole at $z=0$ in \eqref{defM1int},
we see that $M_1(0^+)-M_1(0^-)=-2$. Third, the fact that $M_1(u)$ satisfies
Vign\'eras' equation \eqref{Vigdif} with $\lambda=0$ on $\IR\backslash\{0\}$
is easily seen by acting with the differential operator $\partial_u^2+2\pi u\partial_u$
on the integrand in \eqref{defM1int},
\be
(\partial_u^2+2\pi u\partial_u) \, e^{-\pi z^2-2\pi\I u z} = 2\pi z\, \partial_z  e^{-\pi z^2-2\pi\I u z} ,
\ee
which becomes a total derivative after dividing by $z$. Its shadow
$\tfrac{\I}{4}u\partial_u M_1(u)=\tfrac{\I}{2}\, u\, e^{-\pi u^2}$ is also easily computed
by acting with $\tfrac{\I}{4}u\partial_u$ on the integrand in \eqref{defM1int}.

Similarly, it will be useful to note that the error function $E_1(u)$
can be represented as
\be
\label{defE1int}
E_1(u) = \int_{\IR}  e^{-\pi(u-u')^2} \sign(u')\, \de u' \, ,
\ee
upon using $\int_{-\infty}^u e^{-\pi u'^2} \de u'= \tfrac12\(1+\Erf(u \sqrt{\pi})\)$.
Eq. \eqref{defE1int} means that $E_1(u)$ is the image of the function $\sign(u)$
under the heat kernel operator $e^{\frac{1}{4\pi}\partial_u^2}$. This  makes it manifest  that $E_1$ is a $C^\infty$
function on $\IR$ which asymptotes to $\sign(u)$ as $|u|\to\infty$.
The fact that $E_1(u)$ is a solution of Vign\'eras' equation on $\IR$ with $\lambda=0$
with the same shadow as $M_1$ is also easily seen by acting with the differential
operators $\partial_u^2+2\pi u\partial_u$ and $\tfrac{\I}{4}u\partial_u$ on the integrand
in \eqref{defE1int}.

To show that the integrals \eqref{defM1int} and \eqref{defE1int} satisfy $E_1(u)=M_1(u)+\sign(u)$,
we first move the contour $\ell$ in \eqref{defM1int} towards the real axis, such
that it avoids the pole at $z=0$ from below when $u>0$,
or from above when $u<0$. Equivalently, we write
\be
\label{defM1int3}
M_1(u) = \frac{\I}{\pi}\, \lim_{\epsilon\to 0+}  \int_{\ell}
e^{-\pi z^2 -2\pi \I z u}\,  \frac{\de z}{z-\I\epsilon \sign u}\ .
\ee
Using $ \mathop{\lim}\limits_{\epsilon\to0^+} \frac{1}{z-\I\epsilon}=\Pv(1/z)
+\I\pi \delta(z)$, where $\Pv(1/z)=\mathop{\lim}\limits_{\epsilon\to 0^+} \frac{z}{z^2+\epsilon^2}$
is the principal value, we see that
\be
\label{defM1int4}
M_1(u)+\sign(u) =  \frac{\I}{\pi} \int_{\IR} \Pv(1/z)\, e^{-\pi z^2 -2\pi \I z u} \,  \de z\, .
\ee
Using the fact that $\Pv(1/z)$ is the Fourier transform of $-\I \pi  \sign u$,  we recognize
the right-hand side of \eqref{defM1int4} as $E_1(u)$.
The representations \eqref{defM1int} and \eqref{defE1int} will be key for generalizing
the error functions $M_1$ and $E_1$ relevant for Lorentzian theta series
to the case of signature $(n_+,n_-)$ with $n_+>1$.

\section{Double error functions \label{sec_m2e2}}
\label{derror}

In this section we construct special functions $E_2(\alpha;u_1,u_2)$ and $M_2(\alpha;u_1,u_2)$,
analogous to the error and complementary error functions  $E_1(u)$ and $M_1(u)$ of the previous section,
which satisfy Vign\'eras' equation with $\lambda=0$ on $\IR^2$ (away from codimension one loci, in the case of $M_2$).
The parameter $\alpha$ controls the angle between the two lines where
$M_2$ jumps. We then promote these functions to
solutions $E_2(C_1,C_2;x)$ and $M_2(C_1,C_2;x)$ of Vign\'eras' equation  on $\IR^{2,n-2}$,
parametrized by pairs of vectors $C_1,C_2$ spanning a time-like two-plane.

\subsection{Double error functions}

In this subsection we define the double error functions $M_2$ and $E_2$ and,
in  Proposition \ref{EManalytic}, establish their main  analytic and asymptotic properties.
Proposition \ref{vign2n-2} proves that $E_2$ and $M_2$ are solutions of Vign\'eras' equation,
while Proposition \ref{EMem} expresses $M_2$ and $E_2$ in terms of new functions $m_2$ and $e_2$
which are convenient for analytic and numerical studies.
Theorem \ref{thmkerE2} gives an alternative definition of $E_2$ which generalizes naturally
to higher dimensions. Two-parameter versions of $M_2$ and $E_2$ are briefly mentioned
at the end.

\begin{definition}
\label{DefM2}
Let $\alpha\in \IR$,
$(u_1,u_2)\in \IR^2$, $u_1\neq 0$, $u_2\neq \alpha u_1$.
The `complementary double error function' $M_2(\alpha;u_1,u_2)$ is defined by the absolutely convergent integral
\be
\label{defM2}
M_2(\alpha;u_1, u_2) :=-\frac{1}{\pi^2}
\int_{\mathbb{R}-\I u_1} \left(
 \int_{\mathbb{R}-\I u_2}\,
\frac{e^{-\pi z_1^2 - \pi z_2^2 -2\pi \I (u_1 z_1 + u_2 z_2)}}{z_1(z_2-\alpha z_1)}\,  \de z_2 \right)  \de z_1\ .
\ee
\end{definition}

\begin{definition}
\label{DefE2}
Let $\alpha\in \IR$,
$(u_1,u_2)\in \IR^2$. For $u_1\neq 0$ and $u_2\neq \alpha u_1$, the `double error function' $E_2(\alpha;u_1, u_2)$
is defined in terms of $M_2(\alpha;u_1, u_2)$ by
\be
\label{defE2}
\begin{split}
E_2(\alpha;u_1, u_2) :=&\, M_2(\alpha;u_1, u_2)  +\sign(u_1)\, M_1\!\left(u_2 \right)
+\sign(u_2-\alpha u_1)\, M_1\!\left(\tfrac{u_1+\alpha u_2}{\sqrt{1+\alpha^2}}\right)
\\
&\, + \sign(u_2)\,  \sign(u_1+ \alpha u_2) .
\end{split}
\ee
\end{definition}

\begin{proposition} The functions $M_2$ and $E_2$ satisfy:
\label{EManalytic}
\begin{enumerate}
\item[i)] For any  $\alpha\in \mathbb{R}$, $M_2(\alpha;u_1,u_2)$ is a real valued $C^\infty$ function on $\IR^2$ away from the loci
$u_1=0$ and $u_2=\alpha u_1$. It is discontinuous across these  loci,
with jumps given by
\be
\label{M2xdisc}
\begin{split}
M_2(\alpha;u_1, u_2) \sim & \, -\sign(u_1)\, M_1(u_2)\quad \mbox{as}\quad u_1\to 0,
\\
M_2(\alpha;u_1, u_2) \sim &\, -\sign(u_2-\alpha u_1)\, M_1\!\left(\tfrac{u_1+\alpha u_2}{\sqrt{1+\alpha^2}}\right)
\quad
\mbox{as} \quad u_2-\alpha u_1\to 0.
\end{split}
\ee
\bp{where $\sim$ indicates that the difference between the left and
  right-hand side is continuous} \bp{in a neighbourhood of the respective loci.} 
In contrast, $E_2(\alpha;u_1,u_2)$ extends to a continuous $C^\infty$ function on $\IR^2$.
\item[ii)] For $\alpha=0$, $M_2$  and $E_2$ factorize into products of $M_1$ and $E_1$ respectively,
\be
\label{M2lim}
M_2(0;u_1,u_2)= M_1(u_1)\, M_1(u_2) ,
\qquad
E_2(0;u_1,u_2) = E_1(u_1) \, E_1(u_2).
\ee
\item[iii)] For large $(u_1,u_2)$, $M_2(\alpha;u_1,u_2)$ is exponentially \bp{decreasing},   and behaves as
\be
\label{M2largex}
M_2(\alpha;u_1, u_2) \sim -\frac{e^{-\pi(u_1^2+u_2^2)}}{\pi^2 u_1(u_2-\alpha u_1)}\, .
\ee
In contrast,  $E_2(\alpha;u_1,u_2)$ is locally constant at infinity,
\be
\label{E2largex}
E_2(\alpha;u_1, u_2) \sim \sign(u_2)\,  \sign(u_1+\alpha u_2) .
\ee

\end{enumerate}
\end{proposition}

\begin{proof}\noindent
\begin{enumerate}
\item[i)]To prove \eqref{M2xdisc}, we make the change of variables
$v_1=\tfrac{z_1}{u_1}+\I$ and $v_2=\tfrac{z_2-\alpha z_1}{u_2-\alpha u_1}+\I$ in \eqref{defM2}.
This brings $M_2$ to the form
\be
\label{M2intuv}
\begin{split}
M_2(\alpha;u_1, u_2)=&-\frac{1}{\pi^2} \,\sign(u_1)\,\sign(u_2-\alpha u_1)\,
e^{-\pi  u_1^2-\pi u_2^2}\\
&\times\int_{\mathbb{R}^2} 
e^{-\pi u_1^2 v_1^2-\pi [\alpha u_1 v_1+v_2(u_2-\alpha u_1)]^2}\, \frac{\de v_1\, \de v_2}{(v_1-\I)(v_2-\I)}\, ,
\end{split}
\ee
from which \eqref{M2xdisc} follows. The \bp{second and third} terms in \eqref{defE2} ensure
that $E_2(\alpha;u_1,u_2)$  is continuous across the loci $u_1=0$ and $u_2=\alpha u_1$.
The apparent discontinuities in  \eqref{defE2} on the loci $u_2=0$ and $u_1=-\alpha u_2$
also cancel due to
the fact that $M_1(x)\sim -\sign(x)$ as $x\to 0$. The partial derivatives of $M_2(\alpha;u_1,u_2)$
are obtained by acting with $\partial_{u_1}$ and
$\partial_{u_2}$ on \eqref{defM2} and using \eqref{defM1int},
\be
\label{M2deriv}
\begin{split}
\partial_{u_1} M_2(\alpha;u_1, u_2) =&\, \tfrac{2}{\sqrt{1+\alpha^2}} \,
e^{-\frac{\pi(u_1+\alpha u_2)^2}{1+\alpha^2}} M_1\!\left( \tfrac{u_2-\alpha u_1}{\sqrt{1+\alpha^2}}\right),
\\
\partial_{u_2} M_2(\alpha;u_1, u_2) =&\, 2 e^{-\pi u_2^2} M_1(u_1)
+ \tfrac{2\alpha}{\sqrt{1+\alpha^2}}\,
e^{-\frac{\pi(u_1+\alpha u_2)^2}{1+\alpha^2}} M_1\!\left( \tfrac{u_2-\alpha u_1}{\sqrt{1+\alpha^2}}\right).
\end{split}
\ee
The partial derivatives of $E_2(\alpha;u_1,u_2)$, computed from \eqref{defE2} and \eqref{M2deriv},
\be
\label{E2deriv}
\begin{split}
\partial_{u_1} E_2(\alpha;u_1, u_2) =&\, \tfrac{2}{\sqrt{1+\alpha^2}} \,
e^{-\frac{\pi(u_1+\alpha u_2)^2}{1+\alpha^2}} E_1\!\left( \tfrac{u_2-\alpha u_1}{\sqrt{1+\alpha^2}}\right),
\\
\partial_{u_2} E_2(\alpha;u_1, u_2) =&\, 2 e^{-\pi u_2^2} E_1(u_1)
+ \tfrac{2\alpha}{\sqrt{1+\alpha^2}} \,e^{-\frac{\pi(u_1+\alpha u_2)^2}{1+\alpha^2}}
E_1\!\left( \tfrac{u_2-\alpha u_1}{\sqrt{1+\alpha^2}}\right),
\end{split}
\ee
are $C^\infty$, therefore $E_2(\alpha;u_1, u_2)$ extends to a $C^\infty$ function of $(u_1,u_2)\in\IR^2$.
\item[ii)]
This is an immediate consequence of the definitions.
\item[iii)] Eq. \eqref{M2largex} is obtained by
 the saddle point method on the integral \eqref{defM2}. Eq. \eqref{E2largex}
follows from \eqref{M2largex} and from the fact that $M_1(u)$ is exponentially \bp{decreasing} 
as $|u|\to\infty$.

\end{enumerate}
\end{proof}

\begin{proposition}
\label{vign2n-2}
The functions $M_2(\alpha;u_1, u_2)$ and $E_2(\alpha;u_1, u_2)$ are solutions of
Vign\'eras' equation with $\lambda=0$ on $\IR^2$
equipped with the quadratic form
$Q(u)=Q(u_1,u_2)=u_1^2+u_2^2$
in their respective domains of definition,
\be
\label{VigM2}
\begin{split}
&\left[ \pa_{u_1}^2+\pa_{u_2}^2+2\pi (u_1 \partial_{u_1}+u_2\partial_{u_2}) \right] M_2(\alpha;u_1, u_2)=0,
\\
&\left[ \pa_{u_1}^2+\pa_{u_2}^2+2\pi (u_1 \partial_{u_1}+u_2\partial_{u_2}) \right]\, E_2(\alpha;u_1, u_2) =0.
\end{split}
\ee
Their shadows are given by
\be
\label{shadowM2}
\frac{\I}{4}\, (u_1\pa_{u_1}+u_2\pa_{u_2})\, M_2(\alpha;u_1,u_2) = \frac{\I}{2}
\left[ u_2 \, e^{-\pi u_2^2}  M_1(u_1) + \tfrac{u_1+\alpha u_2}{\sqrt{1+\alpha^2}}\,
e^{-\frac{\pi(u_1+\alpha u_2)^2}{1+\alpha^2} }
M_1\!\left(\tfrac{u_2-\alpha u_1}{\sqrt{1+\alpha^2}}\right)
\right],
\ee
\be
\label{shadowE2}
\frac{\I}{4}\, (u_1 \pa_{u_1} + u_2 \pa_{u_2} )\, E_2(\alpha;u_1,u_2) = \frac{\I}{2}
\left[ u_2 \, e^{-\pi u_2^2} \, E_1(u_1) + \tfrac{u_1+\alpha u_2}{\sqrt{1+\alpha^2}}\,
e^{-\frac{\pi(u_1+\alpha u_2)^2}{1+\alpha^2} }\,
E_1\!\left(\tfrac{u_2-\alpha u_1}{\sqrt{1+\alpha^2}}\right)
\right].
\ee
\end{proposition}

\begin{proof}
Acting with Vign\'eras' operator on the integral representation \eqref{defM2} of $M_2$ gives
\be
\begin{split}
& -\frac{2}{\pi} \int_{\ell_1}\int_{\ell_2}\,
\frac{1}{z_1(z_2-\alpha z_1)} 
\left[ \left(z_1 \p_{z_1} + z_2 \p_{z_2}\right) e^{-\pi (z_1^2 + z_2^2) - 2\pi\I(u_1 z_1 + u_2 z_2) }
\right]\,\de z_1 \de z_2
\\ &
= \frac{2}{\pi} \int_{\ell_1}\int_{\ell_2}\,  \left[\p_{z_1}
\left(\tfrac{1}{z_2-\alpha z_1}\right) + \p_{z_2}\left(\tfrac{z_2}{z_1(z_2-\alpha z_1)}\right)\right]
e^{-\pi (z_1^2 + z_2^2) - 2\pi\I(u_1 z_1 + u_2 z_2) } \, \de z_1 \de z_2 =0,
\end{split}
\ee
where the second line follows from the first by partial integration. This proves the first line in \eqref{VigM2}.
The additional terms in \eqref{defE2} are solutions of Vign\'eras' equation with $\lambda=0$, which proves the second line.
The shadows follow from \eqref{M2deriv} and \eqref{E2deriv}.
\end{proof}

We now introduce two functions $m_2(u_1,u_2)$ and $e_2(u_1,u_2)$  which serve
as building blocks for $M_2$ and $E_2$.

\begin{definition}
For $(u_1,u_2)\in\IR^2$, $u_1\neq 0$, we define
\be
\label{defm2}
m_2(u_1,u_2) :=   2 u_2\, \int_{1}^\infty \, e^{-\pi t^2 u_2^2 } \, M_1(t u_1) \, \de t \ .
\ee
\end{definition}

\begin{definition}
For $(u_1,u_2)\in\IR^2$ we define
\be
\label{defe2}
e_2(u_1,u_2) :=   2 u_2 \int_{0}^1 \, e^{-\pi t^2 u_2^2 } \, E_1(t u_1) \de t \ .
\ee
\end{definition}

\begin{proposition}
\label{emanalytic}
The functions $m_2(u_1,u_2)$ and $e_2(u_1,u_2)$ satisfy the following properties:
\begin{itemize}
\item[i)] $m_2$ is a solution of Vign\'eras' equation on $\IR^2$ with $\lambda=0$, $C^\infty$
away from the locus $u_1=0$, odd with respect to either of its arguments, exponentially \bp{decreasing}  at infinity;
near $u_1=0$, it behaves as $m_2(u_1,u_2)\sim \sign(u_1)\, M_1(u_2)$.

\item[ii)] $e_2$ is  a $C^\infty$ solution of Vign\'eras' equation with $\lambda=0$ on $\IR^2$,
odd with respect to either of its arguments, which asymptotes to $e(tu_1,tu_2)\sim
\tfrac{2}{\pi}{\rm Arctan}(u_1/u_2)$ as $t\to+\infty$.

\item[iii)] For small $(u_1,u_2)$, $e_2$ is given by the convergent series
\be
\label{defe2ser}
e_2(u_1,u_2) =
\sum_{k=0}^{\infty} \sum_{\ell=0}^{\infty} \frac{(-\pi)^{k+\ell}u_1^{2k+1} u_2^{2\ell+1}}{k!\,\ell!\, (k+\tfrac12) (k+\ell+1)}.
\ee

\item[iv)] $m_2$ is \bp{bounded} on $\IR^2$ by
\be
\label{m2bound}
|m_2(u_1,u_2)| < \frac{ |u_2|}{\sqrt{u_1^2+u_2^2}}
\,e^{-\pi(u_1^2+u_2^2)}.
\ee

\item[v)] For large $(u_1,u_2)$, $m_2$ has the asymptotic expansion
\be
\label{defm2ser}
 m_2(u_1,u_2)= u_2\, e^{-\pi(u_1^2+u_2^2)}
 \sum_{k=0}^{\infty}  \sum_{\ell=0}^{\infty} \frac{(-1)^{k+\ell+1}\, \Gamma(\ell+\tfrac12)\, (\ell+k)!}
 {\pi^{k+\ell+\tfrac52}\, \ell!\, u_1^{2\ell+1} (u_1^2+u_2^2)^{k+1}}\, .
\ee

\item[vi)] $m_2$ and $e_2$ satisfy
\bea
\label{m2e2}
m_2(u_1,u_2)+e_2(u_1,u_2) &=&\frac{2}{\pi}\, {\rm Arctan}\frac{u_1}{u_2}  + \sign(u_1) \, M_1(u_2)  ,
\\
\label{e21221}
e_2(u_1,u_2)+e_2(u_2,u_1) &=& \bp{E_1(u_1)\, E_1(u_2)} .
\eea

\end{itemize}
\end{proposition}

\begin{proof}\noindent
\begin{enumerate}
\item[i-ii)] One verifies that $m_2$ and $e_2$ satisfy Vign\'eras' equation using
$E'_1(u)=2 e^{-\pi u^2}=M'_1(u)$ (away from $u=0$). The last statement in i) follows
by substituting $M_1(tu_1)\sim -\sign(tu_1)$ in \eqref{defm2} and using
$M_1(u)=-2u\int_1^\infty\, e^{-\pi u^2 v^2}\, \de v $.
The last
statement in ii) can be shown by using $E_1(u)=2u\int_0^1\de v \, e^{-\pi u^2 v^2}$
and exchanging the $t$ and $v$ integrals.
\item[iii)] Eq. \eqref{defe2ser} follows by inserting  the Taylor expansion
$E_1(u)=\sum\limits_{k=0}^{\infty} \frac{(-\pi)^k u^{2k+1}}{ k!\, (k+\tfrac12)}$ in \eqref{defe2}.

\item[iv)] Eq. \eqref{m2bound} follows from the fact that $|M_1(u)|$ is \bp{bounded} on $\IR$ by $e^{-\pi u^2}$
so that
\be
|m_2(u_1,u_2)| < 2|u_2|\, \int_1^{\infty}\, e^{-\pi(u_1^2+u_2^2)t^2}\, \de t
= \tfrac{ |u_2|}{\sqrt{u_1^2+u_2^2}} \Erfc\!\left(\sqrt{\pi(u_1^2+u_2^2)} \right) <
\tfrac{ |u_2|\,e^{-\pi(u_1^2+u_2^2)}}{\sqrt{u_1^2+u_2^2}}\, .
\ee

\item[v)]

Eq. \eqref{defm2ser} follows by inserting the asymptotic expansion
$M_1(u)= -\frac{1}{\pi}\, \sign(u)\, e^{-\pi u^2}$ $\times\sum\limits_{n=0}^{\infty}(-1)^n\Gamma(n+\tfrac12)
 |u \sqrt{\pi} |^{-2n-1}$  in \eqref{defm2}.

\item[vi)] Eq. \eqref{m2e2} follows by using $M_1(tu_1)=E_1(tu_1)-\sign(tu_1)$.
Eq. \eqref{e21221}
is a consequence of \eqref{defe2ser}.

\end{enumerate}
\end{proof}

\begin{proposition}
\label{EMem}
$M_2(\alpha;u_1,u_2)$ and $E_2(\alpha;u_1,u_2)$ can be expressed in terms of
$m_2$ and $e_2$ via
\be
\label{M2m2}
M_2(\alpha;u_1,u_2) = -m_2(u_1,u_2) - m_2\!\left( \tfrac{u_2-\alpha u_1}{\sqrt{1+\alpha^2}}, \tfrac{u_1+\alpha u_2}{\sqrt{1+\alpha^2}} \right),
\ee
\be
\label{E2e2}
E_2(\alpha;u_1,u_2) =e_2(u_1,u_2)+ e_2\!\left( \tfrac{u_2-\alpha u_1}{\sqrt{1+\alpha^2}}, \tfrac{u_1+\alpha u_2}{\sqrt{1+\alpha^2}} \right)
+ \frac{2}{\pi}\, \Arctan \alpha \, .
\ee
\end{proposition}

\begin{proof} \eqref{M2m2} follows by integrating
the differential equation \eqref{shadowM2} along a radial ray $t(u_1,u_2)$,
$t\in [1,\infty[$, and using the fact that $M_2(\alpha;tu_1,tu_2)$ is exponentially
\bp{decreasing}  as $t\to\infty$. Similarly, integrating \eqref{shadowE2} along a radial ray
and using \eqref{E2largex}
leads to
\be
\label{E2e2h}
E_2(\alpha;u_1,u_2) =  -\tilde e_2(u_1,u_2)
- \tilde e_2\!\left( \tfrac{u_2-\alpha u_1}{\sqrt{1+\alpha^2}},
\tfrac{u_1+\alpha u_2}{\sqrt{1+\alpha^2}} \right) +\sign(u_2)\, \sign(u_1+\alpha u_2),
\ee
where
\be
\label{defe2t}
\tilde e_2(u_1,u_2):=   2 u_2\, \int_{1}^\infty e^{-\pi t^2 u_2^2 } \, E_1(t u_1) \, \de t 
= \frac{2}{\pi} \Arctan\, \frac{u_1}{u_2} - e_2(u_1,u_2).
\ee
This reduces to \eqref{E2e2} upon using the identity
\be
\label{Arctanrule}
 {\rm Arctan}\,\tfrac{u_2-\alpha u_1}{u_1+\alpha u_2}+
{\rm Arctan}\,\tfrac{u_1}{u_2}+  {\rm Arctan}\,\alpha
=\tfrac{\pi}{2}\,  \sign(u_2)\sign(u_1+\alpha u_2).
\ee
\end{proof}

\begin{remark}
Eq. \eqref{M2m2}, \eqref{E2e2}, \eqref{m2e2}, \eqref{defe2ser}
provide an efficient way of evaluating $M_2$ and $E_2$ numerically.
For illustration, plots of $m_2(u_1,u_2)$, $e_2(u_1,u_2)$, $M_2(\alpha;u_1,u_2)$ and $E_2(\alpha;u_1,u_2)$ for $\alpha=1$
are shown in Figures \ref{fig-plotsme} and \ref{fig-plotsME}.
\end{remark}

As an immediate consequence of Propositions \ref{emanalytic} i) and \ref{EMem}, we state the following corollary.
\begin{corollary}
\label{cordiscrete}
$M_2(\alpha;u_1,u_2)$ and $E_2(\alpha;u_1,u_2)$ satisfy the discrete symmetries
\bea
\label{M2discrete}
M_2(\alpha;u_1,u_2) &=& - M_2(-\alpha;-u_1,u_2) = -M_2(-\alpha;u_1,-u_2) = M_2(\alpha;-u_1,-u_2)
\nn\\
&=&M_2\!\left( \alpha; \tfrac{u_2-\alpha u_1}{\sqrt{1+\alpha^2}},\tfrac{u_1+\alpha u_2}{\sqrt{1+\alpha^2}}\right),
\\
\label{E2discrete}
E_2(\alpha;u_1,u_2) &=& - E_2(-\alpha;-u_1,u_2) = -E_2(-\alpha;u_1,-u_2) = E_2(\alpha;-u_1,-u_2)
\nn\\
&=&E_2\!\left( \alpha; \tfrac{u_2-\alpha u_1}{\sqrt{1+\alpha^2}},\tfrac{ u_1+\alpha u_2}{\sqrt{1+\alpha^2}}\right).
\eea
\end{corollary}

The following theorem gives an alternative definition of the double error function $E_2(\alpha;u_1,u_2)$,
which naturally suggests the higher-dimensional generalization
discussed in Section \ref{sec_gen}:

\begin{theorem} \label{thmkerE2}
$E_2(\alpha;u_1,u_2)$ is obtained by acting with the heat kernel operator
$e^{\tfrac{1}{4\pi} (\partial_{u_1}^2+\partial_{u_2}^2)}$ on the locally constant function
$\sign(u_2) \sign(u_1+\alpha u_2)$. Equivalently, it admits the integral representation
\be
\label{defE2int}
\begin{split}
E_2(\alpha;u_1,u_2) =
\int_{\IR^2} \, e^{-\pi(u_1-u'_1)^2-\pi(u_2-u'_2)^2} \sign(u'_2) \, \sign(u'_1+\alpha u'_2)\, \de u'_1 \de u'_2 \ .
\end{split}
\ee
\end{theorem}

\begin{proof}
The right-hand side of \eqref{defE2int} has the same asymptotic behavior \eqref{E2largex}
as $E_2(\alpha;u_1,u_2)$ in the limit where
$(u_1,u_2)$ goes to infinity along a fixed ray. Its shadow can be computed easily
using  \eqref{defE1int}
and matches that of $E_2(\alpha;u_1,u_2)$. Two $C^\infty$ functions with the same shadow
and same asymptotics along radial rays are necessarily equal.
\end{proof}

\begin{remark}
The integral representations \eqref{defM2} and \eqref{defE2int} are related by Fourier transform
over $u'_1$ and $u'_2$, similarly to the relation between \eqref{defM1int} and \eqref{defE1int}.
\end{remark}

\begin{remark}
The operator $e^{\tfrac{1}{4\pi} (\partial_{u_1}^2+\partial_{u_2}^2)}$ is a special case of
the operator $e^{-\tfrac{\Delta}{4\pi}}$ introduced by Borcherds \cite{0919.11036}
in order to construct Siegel
theta functions depending on homogeneous polynomial $P(x_+,x_-)$ of degree $(m_+,m_-)$  (see Remark \ref{rk_Borcherds}).
More generally, for any  {\it locally} polynomial, homogeneous function $P(x_+,x_-)$
of degree $(m_+,m_-)$, the kernel
$\Phi(x)=\bigl[ e^{-\tfrac{\Delta}{4\pi}}P(x_+,x_-)  \bigr]e^{-\pi Q(x_+)}$ is a
$C^{\infty}$ solution of Vign\' eras' equation with
$\lambda=-n_+ +m_- -m_+$, and thus gives rise to a generalized Siegel-Narain theta
series  $\vartheta_{\bfmu}[\Phi]$ of weight $\tfrac{n_--n_+}{2}+m_- -m_+$.
\end{remark}

Finally, we briefly consider a natural two-parameter generalization of \eqref{defM2},
\be
\label{defM2ab}
M_2(\alpha,\beta;u_1, u_2):=  \frac{\alpha-\beta}{\pi^2}
\int_{\mathbb{R}-\I u_1} \left(  \int_{\mathbb{R}-\I u_2}\,
\frac{e^{-\pi (z_1^2 +z_2^2) -2\pi \I (u_1 z_1 + u_2 z_2)}}{(z_1-\alpha z_2)(z_1-\beta z_2)}
 \de z_2\, \right)\de z_1.
\ee
Using $\tfrac{\alpha}{z_1-\alpha z_2}-\tfrac{\beta}{z_1-\beta z_2} = \tfrac{(\alpha-\beta) z_1}{(z_1-\alpha z_2)(z_1-\beta z_2)}\,$,
we can express \eqref{defM2ab} in terms of the one-parameter $M_2$,
\be
\label{M2abtoM2}
M_2(\alpha,\beta;u_1, u_2) = M_2(1/\alpha;u_1, u_2) - M_2(1/\beta;u_1, u_2)
=M_2\!\left( \tfrac{1+\alpha\beta}{\alpha-\beta};
\tfrac{u_1-\beta u_2}{\sqrt{1+\beta^2}},\tfrac{u_2+\beta u_1}{\sqrt{1+\beta^2}}\right).
\ee
Similarly, the natural two-parameter generalization of \eqref{defE2}
\be
\label{defE2ab}
E_2(\alpha,\beta;u_1,u_2) := \sign(\alpha-\beta)\,
\int_{\IR^2} \, e^{-\pi(u_1-u'_1)^2-\pi(u_2-u'_2)^2} \sign(u'_2+\alpha u'_1) \, \sign(u'_2+\beta u'_1)\, \de u'_1 \de u'_2 
\ee
can be shown to satisfy \eqref{M2abtoM2} with $M_2$ replaced by $E_2$.

\begin{figure}
\centerline{
\hfill
\includegraphics[height=4.5cm]{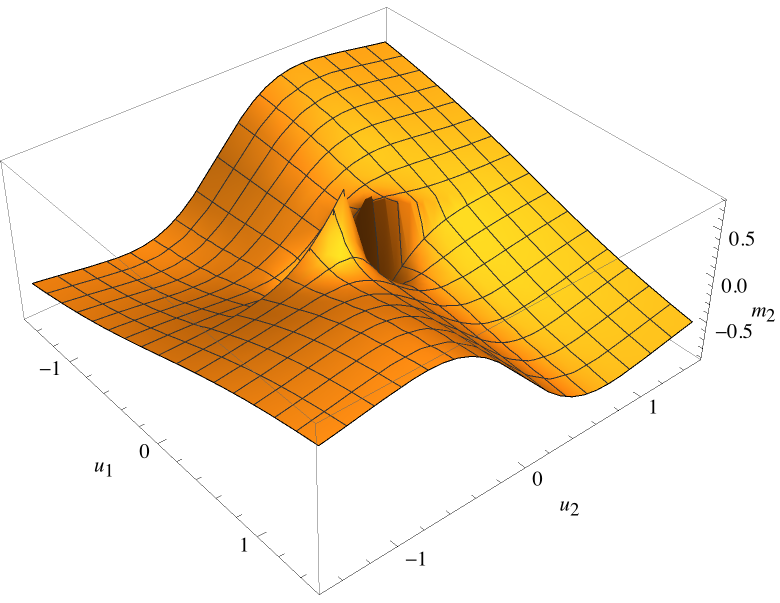}
\hfill\hfill
\includegraphics[height=4.5cm]{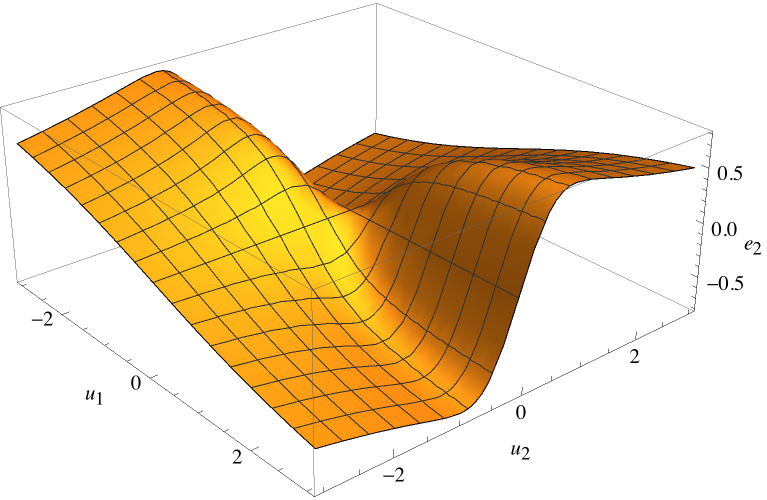}
\hfill
}
\caption{Plots of the functions $m_2(u_1,u_2)$ (left) and $e_2(u_1,u_2)$ (right).}
\label{fig-plotsme}
\end{figure}

\begin{figure}
\centerline{
\hfill
\includegraphics[height=4.5cm]{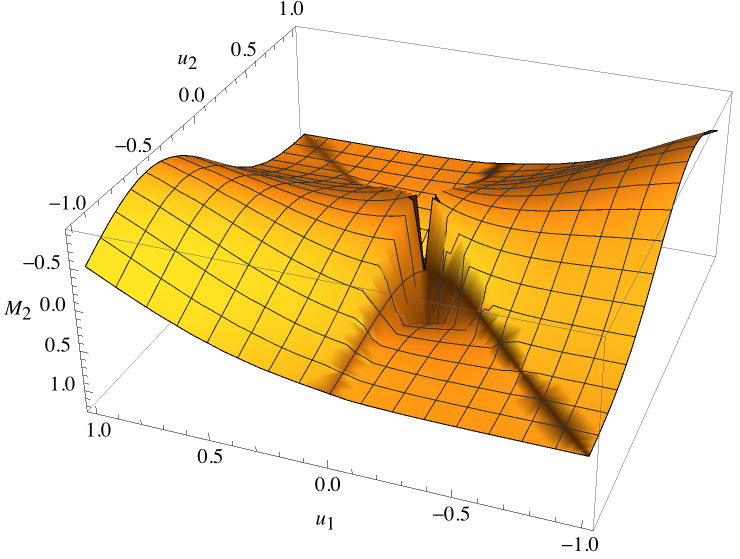}
\hfill\hfill
\includegraphics[height=4.5cm]{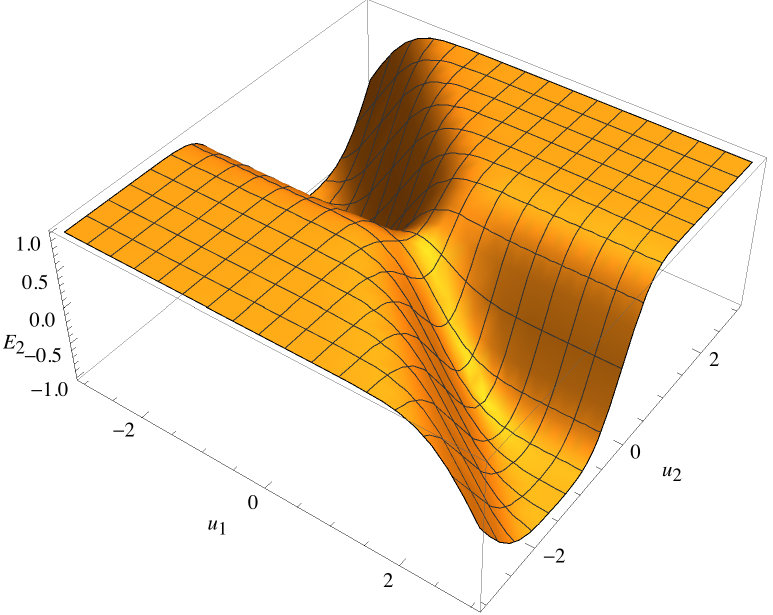}
\hfill
}
\caption{Plots of the functions $M_2(1;u_1,u_2)$ (left) and $E_2(1;u_1,u_2)$ (right).}
\label{fig-plotsME}
\end{figure}

\subsection{Boosted error functions}

We shall now use the functions $E_2(\alpha;u_1,u_2)$ and $M_2(\alpha;u_1,u_2)$ to
construct solutions $E_2(C_1,C_2;x)$ and $M_2(C_1,C_2;x)$ of Vign\'eras' equation on $\IR^{2,n-2}$
(away from suitable loci in the case of $M_2$),
parametrized by a pair of linearly independent vectors $C_1,C_2$ spanning a time-like two-plane.
The latter condition means that
\be
\label{Delta12pos}
Q(C_1)>0,
\qquad
Q(C_2)>0,
\qquad
\Delta(C_1,C_2):=  Q(C_1) Q(C_2) - B(C_1,C_2)^2> 0 .
\ee

Let us introduce some useful notations. We define the vectors
\be
\label{defperp}
C_{1\perp 2} := C_1-\tfrac{B(C_1,C_2)}{Q(C_2)}\, C_2,
\qquad
C_{2\perp 1} := C_2-\tfrac{B(C_1,C_2)}{Q(C_1)}\,C_1,
\ee
such that $B(C_{1\perp 2},C_2)=B(C_{2\perp 1},C_1)=0$, and the linear forms
\be
\label{u12}
u_1(x) :=   \tfrac{B(C_{1\perp 2},x)}{\sqrt{Q(C_{1\perp 2})}}\, ,
\qquad
u_2(x) :=  \tfrac{B(C_2,x)}{\sqrt{Q(C_2)}}\, ,
\ee
such that
\be
\label{u12alpha}
\tfrac{u_2(x)-\alpha u_1(x)}{\sqrt{1+\alpha^2}} =
\tfrac{B({ C_{2\perp 1}},x)}{\sqrt{Q({C_{2\perp 1}})}}\, ,
\qquad
\tfrac{u_1(x)+\alpha u_2(x)}{\sqrt{1+\alpha^2}} =
 \tfrac{B(C_1,x)}{\sqrt{Q(C_1)}}\, ,
\ee
where $\textstyle \alpha=\frac{B(C_1,C_2)}{\sqrt{\Delta(C_1,C_2)}}$. To motivate these
definitions, we note that in a basis where the quadratic form $Q(\bfx)=x_1^2 + x_2^2$
$- \sum_{i=3\dots  n} x_i^2$ is diagonal and the vectors $C_1, C_2$ are chosen as
\be
C_1 = \tfrac{t_1}{\sqrt{1+\alpha^2}}\, (1, \alpha , 0, \dots),
\qquad
C_2=t_2 \, (0,1,0,\dots), \qquad t_1,t_2\in \IR^+,
\ee
the linear forms $u_1(x)$ and $u_2(x)$ reduce to $x_1$ and $x_2$. The
function $E_2(\alpha;u_1(x),u_2(x))$ is then a solution of Vign\'eras' equation on
$\IR^{2,n-2}$ which asymptotes to $ \sign[u_1(x)+\alpha u_2(x)]$ $\sign[u_2(x)] =
\sign B(C_1,x) \sign B(C_2,x)$.

\begin{definition}
\label{DefE2C}
Let $C_1, C_2$ be a  pair of time-like vectors  with $\Delta(C_1,C_2)>0$. We define
the boosted double error function $E_2(C_1,C_2;x)$ by
\be
\label{defE2C}
E_2(C_1,C_2;x) := E_2\!\left( \tfrac{B(C_1,C_2)}{\sqrt{\Delta(C_1,C_2)}};
\tfrac{B(C_{1\perp 2},x)}{\sqrt{Q(C_{1\perp 2})}}\, ,
\tfrac{B(C_2,x)}{\sqrt{Q(C_2)}} \right).
\ee
\end{definition}

\begin{proposition}
\label{PropE2C}
The function $E_2$ satisfies:
\begin{enumerate}
\item[i)]
$E_2(C_1,C_2;x)$ is a $C^\infty$ function on $\IR^{2,n-2}$,
invariant under independent rescalings of $C_1$ and $C_2$ by a positive real number,
invariant under $C_1\leftrightarrow C_2$, and odd under $C_1\mapsto-C_1$ or $C_2\mapsto-C_2$.

\item[ii)] In the limit where $C_1$ and $C_2$ are orthogonal, $E_2(C_1,C_2;x)$ factorizes
\be
\label{E2lim0}
\lim_{B(C_1,C_2)\to 0} E_2(C_1,C_2;x)=  E_1(C_1;x)\,
E_1(C_2;x).
\ee

\item[iii)] In the asymptotic region where $|B(C_1,x)|\to\infty$, keeping $B(C_2,x)$ finite,
\be
E_2(C_1,C_2;x)\sim \sign B(C_1,x)\, E_1(C_2;x).
\ee

\item[iv)] In the asymptotic region where both $|B(C_1,x)|,|B(C_2,x)|\to \infty$,
\be
\label{E2Clargex}
E_2(C_1,C_2;x)\sim \sign B(C_1,x)\, \sign B(C_2,x).
\ee

\item[v)] $E_2(C_1,C_2;x)$ satisfies Vign\'eras' equation with $\lambda=0$ on $\IR^{2,n-2}$,
\be
\label{VigdifE2C}
\left[ B^{-1}(\partial_{\bfx},\partial_{\bfx})   + 2\pi  \bfx\pa_{\bfx}  \right] E_2(C_1,C_2;x) = 0 ,
\ee
and its shadow is given by
\be
\label{E2Cshadow}
 \frac{\I}{2}
\left[
\tfrac{B(C_2,x)}{\sqrt{Q(C_2)}} e^{-\frac{\pi B(C_2,x)^2}{Q(C_2)}}\,
E_1(C_{1\perp 2};x)
+ \tfrac{B(C_1,x)}{\sqrt{Q(C_1)}} e^{-\frac{\pi B(C_1,x)^2}{Q(C_1)}}\,
E_1(C_{2\perp 1};x)
\right].
\ee
\item[vi)] $E_2(C_1,C_2;x)$ admits the integral representation
\be
\label{defE2Cint}
E_2(C_1,C_2;x) = \int_{\left<C_1,C_2\right>} 
\sign B(C_1,y)\, \sign B(C_2,y)\, e^{-\pi Q(y-x_+)}\, \de^2 y\,
\ee
where $\de^2 y$ is the uniform measure on the two-plane $\left<C_1,C_2\right>$ spanned by
$\{C_1,C_2\}$, normalized such that
$\int_{\left<C_1,C_2\right>} e^{-\pi Q(y)}  \,\de^2 y =1$, and
\be
\label{defx12}
x_+= \tfrac{B(C_1,x) Q(C_2)-B(C_2,x) B(C_1,C_2)}{\Delta(C_1,C_2)}\, C_1 +
\tfrac{B(C_2,x) Q(C_1)-B(C_1,x) B(C_1,C_2)}{\Delta(C_1,C_2)}\, C_2
\ee
is the orthogonal projection of $x$ on $\left<C_1,C_2\right>$.

\end{enumerate}
\end{proposition}

\begin{proof}\noindent
\begin{enumerate}
\item[i)] $E_2(C_1,C_2;x)$ is a $C^\infty$ function on $\IR^{2n-2}$ since by
Prop. \ref{EManalytic} i) $E_2(\alpha;u_1,u_2)$ is a $C^\infty$ function on $\IR^2$.
The symmetries of $E_2(C_1,C_2;x)$ under $C_1\leftrightarrow C_2$, $C_1\mapsto-C_1$
and $C_2\mapsto-C_2$ follow from Corollary \ref{cordiscrete}.

\item[ii)] This is a direct consequence of Prop. \ref{EManalytic} ii).

\item[iii)] This follows from the Definitions \ref{DefE2} and \ref{DefE2C},
and the asymptotic behavior of $M_1$ and $M_2$ (Prop. \ref{EManalytic} iii)).

\item[iv)] This is a direct consequence of the previous point and the asymptotic behavior of $M_1$.

\item[v)] Given Definition \ref{DefE2C} together with the relations \eqref{u12} and \eqref{u12alpha},
we have
\be
\begin{split}
&x\partial_x E_2(C_1,C_2;x)=\left[u_1\partial_{u_1}+u_2\partial_{u_2}\right] E_2(\alpha;u_1,u_2),
\\
&B^{-1}(\partial_x,\partial_x) E_2(C_1,C_2;x)=\left[\partial^2_{u_1}+\partial^2_{u_2}\right] E_2(\alpha;u_1,u_2).
\end{split}
\ee
The claim now follows from Prop. \ref{vign2n-2}.

\item[vi)] This follows from \eqref{defE2int}.
\end{enumerate}
\end{proof}

\begin{definition}
Let $C_1, C_2$ be a  pair of time-like vectors with $\Delta(C_1,C_2)>0$. We define
the boosted complementary double error function $M_2(C_1,C_2;x)$ by
\be
\label{defM2C}
M_2(C_1,C_2;x) := M_2\!\left( \tfrac{B(C_1,C_2)}{\sqrt{\Delta(C_1,C_2)}};
\tfrac{B(C_{1\perp 2},x)}{\sqrt{Q(C_{1\perp 2})}},
\tfrac{B(C_2,x)}{\sqrt{Q(C_2)}} \right).
\ee
\end{definition}

\begin{proposition}
$M_2$ satisfies the following properties:
\begin{enumerate}
\item[i)] $M_2(C_1,C_2;x)$ is a  $C^\infty$ function of $x$ away from the loci $B(C_{1\perp 2},x)=0$
and $B(C_{2\perp 1},x)=0$, invariant under independent rescalings of $C_1$ and $C_2$ by a positive real number,
invariant under $C_1\leftrightarrow C_2$, and odd under $C_1\mapsto -C_1$ or $C_2\mapsto-C_2$.

\item[ii)] In the limit where $C_1$ and $C_2$ are orthogonal, $M_2(C_1,C_2;x)$ factorizes
\be
\label{M2lim0}
\lim_{B(C_1,C_2)\to 0} M_2(C_1,C_2;x)=  M_1(C_1;x)\, M_1(C_2;x).
\ee

\item[iii)] $M_2(C_1,C_2;x)$ is \bp{bounded} by
\be
\label{M2Clargex}
|M_2(C_1,C_2;x)|
<\frac{ \left| \frac{B(C_1,x)}{\sqrt{Q(C_1)}} \right| + \left| \frac{B(C_2,x)}{\sqrt{Q(C_2)}} \right|}{\sqrt{Q(x_+)}}
\ e^{-\pi Q(x_+)},
\ee
where $x_+$ is the  orthogonal projection \eqref{defx12} of $x$ on the plane spanned by $\{C_1,C_2\}$.

\item[iv)] $M_2(C_1,C_2;x)$ is a solution of Vign\'eras' equation with $\lambda=0$ on $\IR^{2,n-2}$
away from the aforementioned loci.

\item[v)] $M_2(C_1,C_2;x)$ has the integral representation
\be
\label{defM2Cint}
M_2(C_1,C_2;x) =- \frac{1}{\pi^2}\,
\int_{\left<C_1,C_2\right>-\I x_+} \frac{ \sqrt{\Delta(C_{1\perp2},C_{2\perp1})}}
{B(C_{1\perp2},z)\, B(C_{2\perp1},z)}\, e^{-\pi Q(z)  -2\pi \I B(x,z)} \,\de^2 z,
\ee
where $\de^2 z$ is the uniform measure on the two-plane $\left<C_1,C_2\right>$, normalized
as indicated below \eqref{defE2Cint}, and $\Delta(C_{1\perp2},C_{2\perp1})
=Q(C_{1\perp2}) Q(C_{2\perp1})- B(C_{1\perp2},C_{2\perp1})^2$.
\end{enumerate}
\end{proposition}

\begin{proof}\noindent
\begin{enumerate}
\item[i-ii)] These statements follow from the corresponding statements about $M_2(\alpha;u_1,u_2)$,
analogously to the proofs of Prop. \ref{PropE2C} i-ii).
\item[iii)] This statement follows from \eqref{m2bound} and \eqref{M2m2}, upon noting that
\be
\tfrac{\pi B(C_1,x)^2}{Q(C_1)}
+\tfrac{\pi B(C_{2\perp 1},x)^2}{Q(C_{2\perp 1})} =
\tfrac{Q(C_1) B(C_2,x)^2+Q(C_2)\, B(C_1,x)^2 - 2 B(C_1,C_2) B(C_1,x) B(C_2,x)}{\Delta(C_1,C_2)} =  Q(x_+).
\ee
\item[iv)] The proof is identical to the proof of Prop. \ref{PropE2C} v).
\item[v)] Upon decomposing
\be
z= z_1 \tfrac{C_{1\perp2}}{\sqrt{Q(C_{1\perp2})}} + z_2 \tfrac{C_2}{\sqrt{Q(C_2)}}\, ,
\qquad
x= x_1 \tfrac{C_{1\perp2}}{\sqrt{Q(C_{1\perp2})}} + x_2 \tfrac{C_2}{\sqrt{Q(C_2)}}+ x_-
\ee
and using $\Delta(C_{1\perp2},C_{2\perp1})= \frac{\Delta(C_1,C_2)^3}{Q(C_1) Q(C_2)}$,
Eq. \eqref{defM2Cint} reduces to \eqref{defM2} with the same parameters $u,x_1,x_2$ as in \eqref{defM2C}.

\end{enumerate}
\end{proof}

\begin{proposition}
The functions $E_2(C_1,C_2;x)$ and $M_2(C_1,C_2;x)$ are related by
\be
\label{E2CCfromM2}
\begin{split}
E_2(C_1,C_2;x) = &\,  M_2(C_1,C_2;x)+ \sign B(C_1,x)\, \sign B(C_2,x)
\\
& +  \sign B(C_{1\perp 2},x) \, M_1(C_2;x)
+ \sign B(C_{2\perp 1},x)\, M_1(C_1;x)
\\
 = &\,  M_2(C_1,C_2;x)- \sign B(C_{1\perp2},x)\, \sign B(C_{2\perp1},x)
\\
&+  \sign B(C_{1\perp 2},x) \, E_1(C_2;x)
+ \sign B(C_{2\perp 1},x)\, E_1(C_1;x) .
\end{split}
\ee
\end{proposition}

\begin{proof} The first equation follows directly from \eqref{defE2}, \eqref{defE2C}, \eqref{defM2C}.
The second equation follows from the first by using
\be
\begin{split}
 &\sign B(C_1,x)\, \sign B(C_2,x)+ \sign B(C_{1\perp2},x)\, \sign B(C_{2\perp1},x)
 \\
 &=\sign B(C_{1\perp 2},x)\,\sign B(C_{2},x)+\sign B(C_{2\perp 1},x)\, \sign B(C_1,x) ,
 \end{split}
\ee
which itself follows from the sign rule
\be
\label{42sign2}
\begin{split}
&\Bigl[\sign u_1- \sign(u_1+\alpha u_2)\Bigr]\, \Bigl[\sign u_2- \sign(u_2-\alpha u_1)\Bigr]=0
\end{split}
\ee
valid for any $u_1,u_2,\alpha\in \IR^3$.

\end{proof}

\bp{
It will be useful to relax in part the conditions \eqref{Delta12pos}, and define the function $E_2(C_1,C_2;x)$
when either $C_1$ or $C_2$ are null, and $\Delta_{12}\geq 0$.
In particular, the last condition implies that in such special cases $B(C_1,C_2)$ necessarily vanishes.
\begin{itemize}
\item For $Q(C_1)=0$, we define $E_2(C_1,C_2;x)$ by taking the limit $Q(C_1)\to 0$ in \eqref{E2lim0},
\be
\label{E2null}
\begin{split}
E_2(C_1,C_2;x) = \sign B(C_1,x)\,   E_1(C_2;x).
\end{split}
\ee
\item The case $Q(C_1)=Q(C_2)=0$  is obtained by further taking the limit $Q(C_2)\to 0$,
\be
\label{E2nulld}
E_2(C_1,C_2;x) = \sign B(C_1,x) \, \sign B(C_2,x) .
\ee
\item When $C_1$ and $C_2$ are both timelike with $\Delta_{12}=0$, we define $E_2(C_1,C_2;x)$
by 
\be
\label{E2Delta0}
E_2(C_1,C_2;x) =    \sign B(C_1,C_2)
+ \sign B(C_{2\perp 1},x)\, \Bigl[ M_1(C_1;x) - \sign B(C_1,C_2) \,  M_1(C_2;x)\Bigr] .
\ee
\end{itemize}
The last definition follows by noting that 
the parameter
$\alpha=\tfrac{B(C_1,C_2)}{\sqrt{\Delta(C_1,C_2)}}$ in \eqref{defM2C} becomes infinite
in the limit $\Delta(C_1,C_2)\to 0$,
so from \eqref{M2m2} we get $M_2(C_1,C_2;x)=0$. Moreover, taking the limit $\alpha\to
\infty$ in \eqref{u12alpha}, we get $\sign B(C_{1\perp 2},x) = - \sign B(C_1,C_2)\, \sign
B(C_{1\perp 2},x)$ and $\sign B(C_1,x) = \sign B(C_1,C_2)\, \sign B(C_2,x)$, leading
to \eqref{E2Delta0}. In the special case where $C_1$ and $C_2$ are collinear, $C_{2\perp 1}=0$
so the second  term in \eqref{E2Delta0} vanishes, leaving only the first term.
}


\section{Indefinite theta series of signature $(2,n-2)$ \label{sec_confth}}
\label{2n-2}

In this section, we turn to the construction of holomorphic theta series of signature $(2,n-2)$
and their modular completion,
using the special functions $E_2$ and $M_2$ introduced in the previous section.

\subsection{Conformal theta series}
\label{conf_theta}
In analogy with
\eqref{PhiZwegers0}, we consider the locally constant function
\be
\label{prod2sign}
\Phi_2(x) =
\frac14\,\Bigl[ \sign B(C_1,x)-\sign B(C'_1,x) \Bigr] \, \Bigl[\sign B(C_2,x)-\sign B(C'_2,x)\Bigr],
\ee
where the four vectors $C_1, C'_1, C_2, C'_2$ are time-like. Our goal is to find
sufficient conditions on these vectors in order for the theta series $\vartheta_{\bfmu}[\Phi_2,0]$
to be convergent and to admit a modular completion. In order to state our result, we let $C_{j'}=C_j'$ and shall
denote by $\Delta_{\cI}$ the determinant of the Gram matrix $B(C_i,C_j)_{i,j\in\cI}$,
where $\cI$ is a subset of indices $\{ 1,1',2,2'\}$. We furthermore let $D_{ij}$ be off-diagonal
cofactors of the Gram matrix $B(C_i,C_j)_{i,j\in \{1,1',2,2' \} }$.
The conditions for convergence of $\vartheta_{\bfmu}[\Phi_2,0]$ are given by the following theorem.

\begin{theorem}
\label{convergence2n}
Assume that the vectors $C_1, C'_1, C_2, C'_2$ satisfy the following conditions:
\begin{subequations}
\label{cond_thm1}
\bea
\label{cond1}&Q(C_1),\ Q(C'_1), \ Q(C_2),\ Q(C'_2)>0,
\\
\label{cond8}&\Delta_{11'22'}>0,
\\
\label{cond9}&D_{11'},\ D_{22'}\geq 0,
\\
\label{cond10}& {\operatorname M} \mbox{ is negative definite},
\eea
\end{subequations}
where ${\operatorname M}$ is the symmetric matrix
\be
\label{defM}
{\operatorname M}= \begin{pmatrix} \Delta_{1'22'} & 0 & D_{12} & D_{12'}
\\
  0 & \Delta_{122'} & D_{1'2} & D_{1'2'}
\\
  D_{12} & D_{1'2} & \Delta_{11'2'} & 0
\\
  D_{12'}& D_{1'2'} & 0 & \Delta_{11'2} \end{pmatrix}.
\ee
Then the theta series $\vartheta_{\bfmu}[\Phi_2,0]$ with kernel \eqref{prod2sign} is convergent. Moreover, it is holomorphic
in $\tau$ and $z$ (in the sense of Remark \ref{rk_hol}),
away from the loci where $B(C,k+b)=0$ for $C\in\{C_1,C'_1,C_2,C'_2\}$ and $k\in\Lambda+\mu+\tfrac12 p$.
\end{theorem}

\begin{proof}
To establish the convergence of $\vartheta_{\bfmu}[\Phi_2,0]$, we note that \eqref{cond1} and
\eqref{cond8} imply that $\left<C_1,C'_1,C_2,C'_2\right>$ span a signature $(2,2)$ four-plane
in $\IR^{2,n-2}$. Thus, for any $x$ linearly independent from these vectors,
$\left<x, C_1,C'_1,C_2,C'_2\right>$
span a signature $(2,3)$ five-plane, so $\Delta(x,C_1,C'_1,C_2,C'_2)<0$. The latter evaluates to
\be
\label{Gram50mat}
\begin{split}
\Delta(x,C_1,C'_1,C_2,C'_2) = &  \Delta_{11'22'}\, Q_-(x)
-X {\operatorname M} X^t,
\end{split}
\ee
where
\be
Q_-(x) :=   Q(x) -  2\, \frac{B(C_1,x) B(C'_1,x)\, D_{11'}+ B(C_2,x) B(C'_2,x)\, D_{22'}}{ \Delta_{11'22'}}
\ee
and $X$ is the \bp{row} vector  $\bigl(B(C_1,x), B(C'_1,x), B(C_2,x), B(C'_2,x)\bigr)$. Since
by assumption $M$ is negative definite, it follows that
$Q_-(x)<0$. If instead $x$ lies in the four-plane $\left< C_1,C'_1,C_2,C'_2 \right>$,
then $\Delta(x,C_1,C'_1,C_2,C'_2)=0$, so $Q_-(x)$ vanishes if $X=0$, which implies $x=0$.
Thus, the quadratic form $Q_-(x)$ is negative definite.
Now, the only vectors $x$ for which $\Phi_2(x)\neq 0$ are such that $B(C_1,x) B(C'_1,x)\leq 0$ and $B(C_2,x) B(C'_2,x)\leq 0$.
Since $D_{11'}$ and $D_{22'}$ are assumed to be both $\geq 0$ and $\Delta_{11'22'}>0$,
it follows that $Q(x)\leq Q_-(x)$ unless $\Phi_2(x)=0$. Thus, $f(x)=\Phi_2(x)\, e^{\tfrac{\pi}{2}Q(x)}$
is dominated by $e^{\tfrac{\pi}{2}Q_-(x)}$ and so lies in $L^1(\Lambda\otimes\IR)$.
This ensures the convergence of $\vartheta_{\bfmu}[\Phi_2]$. Holomorphy in $\tau$ and $z$
follows from the fact that $\Phi_2(x)$ is locally constant.
\end{proof}

The next theorem gives sufficient conditions for the existence of a modular completion of
$\vartheta_{\bfmu}[\Phi_2,0]$ using the double error function $E_2(C_1,C_2,x)$ defined in \eqref{defE2}.
\begin{theorem}
In addition to the conditions in Theorem \ref{convergence2n},
we require that $C_1, C'_1, C_2, C'_2$ satisfy the following conditions:
\begin{subequations}
\label{cond_thm2}
\bea
\label{cond6}\Delta_{12},\Delta_{1'2},\ \Delta_{12'},\ \Delta_{1'2'} &>& 0  ,
\\
\label{cond2}
B(C_{2\perp 1},C_{2'\perp 1})&=&
B(C_2,C_2')-\tfrac{B(C_1,C_2) B(C_1,C_2')}{Q(C_1) } \ \geq 0,
\\
\label{cond3}
B(C_{1\perp 2},C_{1'\perp2})&=&
 B(C_1,C_1')-\tfrac{B(C_2,C_1) B(C_2,C_1')}{Q(C_2)}
\ \geq 0,
\\
\label{cond4}
B(C_{2\perp 1'},C_{2'\perp 1'})&=&
B(C_2,C'_2)-\tfrac{ B(C'_1,C'_2) B(C'_1,C_2)}{Q(C'_1)}\ \geq 0,
\\
\label{cond5}
B(C_{1\perp 2'},C_{1'\perp2'})&=&
B(C_1,C_1')-\tfrac{B(C'_2,C_1) B(C'_2,C_1')}{Q(C'_2)}\ \geq 0.
\eea
\end{subequations}
Then, the theta series $\vartheta_{\bfmu}[\widehat\Phi_2,0]$  with kernel
\be
\label{prod2signE}
\widehat\Phi_2(x) = \frac14\, \Bigl[ E_2(C_1,C_2;x) - E_2(C_1,C_2';x) - E_2(C'_1,C_2;x)+E_2(C'_1,C'_2;x)\Bigr]
\ee
is a non-holomorphic vector-valued Jacobi form of weight
$(\tfrac{n}{2},0)$. Its shadow  is the non-holomorphic Lorentzian
theta series $\vartheta_{\bfmu}[\Psi_2,-2]$ with kernel
\be
\label{prod2signshadow}
\begin{split}
\Psi_2(x) =& \,\frac{\I}{8}\, \Biggl\{ \tfrac{B(C_1,x)}{\sqrt{Q(C_1)}}\, e^{- \tfrac{\pi B(C_1,x)^2}{Q(C_1)}}
\Bigl[ E_1(C_{2\perp 1};x)-E_1(C_{2'\perp 1};x)\Bigr]
\Biggr.
\\&\,
+\tfrac{B (C_2,x)}{\sqrt{Q(C_2)}}\, e^{- \tfrac{\pi B(C_2,x)^2}{Q(C_2)}}
\Bigl[E_1(C_{1\perp2};x)- E_1(C_{1'\perp2};x)\Bigr]
\\&\,
+
 \tfrac{B(C'_1,x)}{\sqrt{Q(C'_1)}}\,
e^{-\tfrac{\pi B(C'_1,x)^2}{Q(C'_1)}} \Bigl[E_1(C_{2'\perp1'};x)- E_1(C_{2\perp1'};x)\Bigr]
\\
&\, \Biggl.
+ \tfrac{B(C'_2,x)}{\sqrt{Q(C'_2)}}\,
e^{-\tfrac{\pi B(C'_2,x)^2}{Q(C'_2)}} \Bigl[E_1(C_{1'\perp2'};x)- E_1(C_{1\perp2'};x)\Bigr]
  \Biggr\}.
\end{split}
\ee

\end{theorem}

\begin{proof} We first address the convergence of $\vartheta_{\bfmu}[\widehat\Phi_2]$.
Using \eqref{E2CCfromM2}, the difference between \eqref{prod2signE} and \eqref{prod2sign} can be written as
\be
\label{prod2signadd}
\begin{split}
\widehat\Phi_2-\Phi_2 = &\, \frac14\, \Biggl\{  M_2(C_1,C_2;x) - M_2(C_1,C_2';x) - M_2(C'_1,C_2;x)+M_2(C'_1,C'_2;x)
\Biggr. \\
&\, + \Bigl[ \sign B(C_{2\perp1},x) -  \sign B(C_{2'\perp1},x) \Bigr]
M_1(C_1;x)
\\
&\, + \Bigl[  \sign B(C_{1\perp2},x) - \sign B (C_{1'\perp2},x)   \Bigr]
M_1(C_2;x)
\\
&\, + \Bigl[ \sign B(C_{2'\perp1'},x) -  \sign B(C_{2\perp1'},x) \Bigr]
M_1(C'_1;x)
\\
&\, \Biggl. +\Bigl[ \sign B(C_{1'\perp2'},x) - \sign B(C_{1\perp2'},x) \Bigr]
M_1(C'_2;x)
\Biggr\}.
\end{split}
\ee
To see that the theta series based on the first term $M_2(C_1,C_2,x)$ in
\eqref{prod2signadd} is convergent, we use the uniform bound \eqref{M2Clargex}.
The condition (\ref{cond6}) implies that $\left<C_1,C_2\right>$
spans a  signature $(2,0)$ plane. Therefore
$\Delta(x,C_1,C_2)\leq 0$ for any $x$, with the equality saturated when $x$ lies in this plane.
Evaluating the determinant as in \eqref{DxCC}, this
shows that $Q(x)-Q(x_+)$ is negative semi-definite, and that $Q(x)-2Q(x_+)$ is negative definite.
Thus, $f(x):=  M_2(C_1,C_2;x) \, e^{\tfrac{\pi}{2}Q(x)}$
is dominated by $e^{\tfrac{\pi}{2}[Q(x)-2Q(x_+)]}$ so lies in $L^1(\Lambda\otimes \IR)$.
The same argument applies to the  three other terms on the first line of \eqref{prod2signadd}.

For the second line of \eqref{prod2signadd}, we note that the vectors $C_{2\perp1},C_{2'\perp1}$
are time-like vectors in the signature $(1,n-2)$  hyperplane $\left< C_1\right>^\perp$
orthogonal to $C_1$, with positive inner product due to \eqref{cond2}.
The same argument which was used to prove the convergence of \eqref{PhiZwegers0}, shows that
$f(x)=\bigl[ \sign B(C_{2\perp1},x) -  \sign B(C_{2'\perp1},x) \bigr]
e^{\tfrac{\pi}{2} Q_\perp(x)}$ lies in $L^1(\left< C_1\right>^\perp)$,
where $Q_\perp(x)$ is the norm of the projection of
$x$ on $\left< C_1\right>^\perp$. On the other hand,
$\Bigl|M_1(x) e^{\tfrac{\pi x^2}{2}}\Bigr|\leq e^{-\tfrac{\pi x^2}{2}}$ is
in $L_1(\IR)$ on the orthogonal complement $\left< C_1\right>$. Thus,
$f(x)=\bigl[ \sign B(C_{2\perp1},x)-   \sign B(C_{2'\perp1},x)\bigr] M_1(C_1;x)e^{\tfrac{\pi}{2}Q(x)}$
lies in $L^1(\Lambda\otimes \IR)$.
Repeating this reasoning for the other three contributions in \eqref{prod2signadd} and
combining this with the result in the previous paragraph, we see that
the theta series $\vartheta_{\bfmu}[\widehat\Phi_2-\Phi_2]$ is convergent.

Finally, from \eqref{VigdifE2C} we see that the kernel \eqref{prod2signE} satisfies condition ii)
of Vign\'eras' theorem with $\lambda=0$.
It follows that $\vartheta_{\bfmu}[\widehat\Phi_2]$ is a non-holomorphic
modular form of weight $(\tfrac{n}{2},0)$. Its shadow follows easily from \eqref{E2Cshadow}.
\end{proof}

\begin{remark}
A few comments on the conditions \eqref{cond_thm1} and \eqref{cond_thm2} are in order:
\begin{itemize}
\item The condition ${\operatorname M} <0$ implies that all diagonal elements of ${\operatorname M}$ are strictly negative.
Thus, assuming that \eqref{cond1} holds, each of $\left<C_1',C_2,C_2'\right>$, $\left<C_1,C_2,C_2'\right>$,
$\left<C_1,C_1',C_2\right>$, $\left<C_1,C_1',C_2'\right>$ span a signature $(2,1)$ three-plane in $\IR^{2,n-2}$.
The restriction of $B$ on $\left<C_1,C_1',C_2,C_2'\right>$ is then either non-degenerate
of signature $(2,2)$,
in which case $\Delta_{11'22'}>0$, or degenerate, in which
case $\Delta_{11'22'}=0$. The condition \eqref{cond8} ensures that it is non-degenerate.

\item In the special case where $B(C_1,C_2)=B(C_1,C'_2)=B(C'_1,C_2)=B(C'_1,C'_2)=0$, the conditions
\eqref{cond_thm1} and \eqref{cond_thm2} reduce to \eqref{cond1} along with
$B(C_1,C'_1), B(C_2,C'_2)\geq 0$ and  $\Delta_{11'}, \Delta_{22'}<0$.

\item The conditions \eqref{cond_thm1}, \eqref{cond_thm2} all follow from the stronger conditions
\begin{subequations}
\label{cond_alt}
\bea
& Q(C_1),\ Q(C'_1), \ Q(C_2),\ Q(C'_2)>0,\quad B(C_1,C'_1), \ B(C_2,C'_2)\geq 0,
\\
& \Delta_{12},\ \Delta_{1'2},\ \Delta_{12'},\ \Delta_{1'2'} > 0, \quad
\Delta_{11'}, \ \Delta_{22'} < 0 ,
\\
&D_{11'},\ D_{22'}\geq 0,
\\
& {\operatorname M} \mbox{ is negative definite}.
\eea
\end{subequations}
This is easily seen by writing
\be
\begin{split}
\label{Delta122'}
0> \Delta_{122'} =&\, Q(C_2)\,\Delta_{12'}+Q(C_2')\,\Delta_{12}-Q(C_1)\,\Delta_{22'}
\\&\,
-2B(C_2,C_2')\,\left[ Q(C_1) B(C_2,C_2')-B(C_1,C_2)B(C_1,C_2') \right],
\end{split}
\ee
and similarly for $ \Delta_{122'}, \Delta_{11'2'},  \Delta_{11'2} $. However,
it is easy to find solutions of  \eqref{cond_thm1} which violate the conditions
$\Delta_{11'}, \ \Delta_{22'} < 0$ in \eqref{cond_alt}.

\item Numerical searches seem to indicate that the
conditions \eqref{cond_thm2} follow
from \eqref{cond_thm1}, but we have not been able to establish this analytically.

 \end{itemize}
\end{remark}

\begin{remark}
For the case at hand, Eq. \eqref{wPhiperiod} shows that
the non-holomorphic theta series $\vartheta_{\bfmu}[\widehat\Phi_2]$ decomposes into
the sum of the holomorphic theta series $\vartheta_{\bfmu}[\Phi_2]$ and an Eichler integral
of the Lorentzian theta series $\vartheta_{\bfmu}[\Psi_2]$. Both terms transforms anomalously
under modular transformations, but the anomalies cancel in the sum.
\end{remark}

\subsection{Null limit\label{sec_null}}

Let us now extend this construction to the case where one vector, say $C'_1$, degenerates
to a primitive null vector in $\Lambda$, while the other vectors $C_1,C_2,C'_2$ remain time-like.
We assume furthermore that $B(C'_1,C_2)=B(C'_1,C'_2)=0$, in order to ensure that $\Delta_{1'2},\Delta_{1'2'}\geq 0$.
As a result, the determinant $\Delta_{11'22'}$ and cofactors $D_{ij}$ simplify to
\be
\begin{split}
\Delta_{11'22'}=&-B(C_1,C'_1)^2 \Delta_{22'},\quad\\
D_{11'}=&-B(C_1,C'_1) \Delta_{22'}, \quad
D_{22'}=B(C_1,C'_1)^2 B(C_2,C_2'), \\
D_{1'2}=&\,B(C_1,C'_1) [B(C_1,C_2) Q(C'_2)-B(C_1,C'_2) B(C_2,C'_2)],\\
D_{1'2'}=&\,B(C_1,C'_1) [B(C_1,C'_2) Q(C_2)-B(C_1,C_2) B(C_2,C'_2)] ,
\end{split}
\ee
and $\Delta_{1'22'}=D_{12}=D_{12'}=0$. The conditions \eqref{cond_thm1} for the convergence
of $\vartheta_\mu[\Phi_2]$ now become
\begin{subequations}
\label{condl}
\bea
\label{cond1l}&Q(C_1),\,  Q(C_2),\, Q(C'_2)>0,\quad  Q(C'_1)=0,
\\
\label{cond234I}& \Delta_{22'}<0, \quad
B(C_1,C'_1)>0,\quad  B(C_2,C'_2)\geq 0 , 
\\
\label{cond5I}& {\operatorname M}_0 \mbox{ is negative definite},
\eea
\end{subequations}
where ${\operatorname M}_0$ is the lower right $3\times 3$ matrix of ${\operatorname M}$,
\be
\label{defK}
{\operatorname M}_0= \begin{pmatrix}
  \Delta_{122'} & D_{1'2} & D_{1'2'}
\\ D_{1'2} & \Delta_{11'2'} & 0
\\
 D_{1'2'} & 0 & \Delta_{11'2} \end{pmatrix}.
\ee

For later use, we deduce a few inequalities from these conditions.
The $(2,2)$ minor of ${\operatorname M}_0$ equals $-B(C_1,C'_1)^2\Delta_{12}\,\Delta_{22'}$,
and the $(3,3)$ minor equals $-B(C_1,C'_1)^2\Delta_{12'}\,\Delta_{22'}$.
Therefore \eqref{cond5I} implies $\Delta_{12}>0$ and $\Delta_{12'}>0$.
From these inequalities and $\Delta_{122'}<0$ (by \eqref{cond5I}),
we deduce that the first line on the right-hand side of  \eqref{Delta122'}
is negative definite, and therefore $B(C_{2\perp 1}, C_{2'\perp 1})>0$.
Noting that $\Delta_{122'}=Q(C_1)\, \Delta(C_{2\perp 1},C_{2'\perp 1})$,
we see that the condition $\Delta_{122'}<0$ implies that $C_{2\perp 1}$ and $C_{2'\perp 1}$
span a subspace of signature (1,1) inside
$\left< C_1\right>^\perp$.

Following the idea of the proof in \cite[\S 2.2]{Zwegers-thesis},  in $\vartheta[\Phi_2]$ we
decompose $k\in \Lambda+\mu+\tfrac{\bfp}{2}$
into $k=k_0 + m_1 C'_1$ with $m_1\in \IZ$ and $k_0\in \Lambda_{1}$ where
\be
\Lambda_1=\left\{k_0\in \Lambda+\mu+\tfrac{\bfp}{2} \ :\ \tfrac{B(C_1,k_0+b)}{B(C_1,C'_1)}\in [0,1)\right\}.
\label{latticeLC}
\ee
Replacing the sum over $k$ by a sum over $k_0$ and $m_1$, $\vartheta_\mu[\Phi_2]$ takes the form
\be
\begin{split}
\theta_\mu[\Phi_2](\tau,b,c)=&\,
\sum_{m_1\in \IZ}\sum_{\bfk_0\in \Lambda_{1}}
(-1)^{B(k_0,p)+m_1B(C'_1,p)}\, \Phi_2(k_0+m_1 C'_1+b)
\\
&\,  \times \ q^{-\tfrac12 Q(k_0+b)-m_1 B(k_0+b,C_1')} \expe{B(c,k_0+m_1C'_1+\tfrac12\, b)}
\end{split}
\ee
with
\be
\begin{split}
\Phi_2(k_0+m_1C'_1+b)=&\, \frac{1}{4}\,\Bigl[\sign\!\left(B(C_1,k_0+b)+m_1B(C_1,C_1')\right)-\sign B(C_1',k_0+b)\Bigr]
\\
&\, \times
\Bigl[\sign B(C_2,k_0+b)-\sign B(C_2',k_0+b)\Bigr].
\end{split}
\ee
The sum over $m_1$ converges provided  $B(C_1',k_0+b)\neq 0$ for all $k_0\in \Lambda_{1}$. Summing up
the geometric series, we obtain
\be
\label{indefresummed}
\begin{split}
\vartheta_\mu[\Phi_2](\tau,b,c) =&\, \frac12\sum_{\bfk_0\in \Lambda_{1}}
(-1)^{B(\bfk_0,\bfp)}\,\Phi_1(k_0+b)\,q^{-\frac12 Q(\bfk_0+\bfb)}\,\expe{B(\bfc,\bfk_0+\tfrac12\bfb)}
\\
& \, \times \left(\frac{1}{1-\expe{B(C'_1, \bfc-\tau(\bfk_0+b)+\frac{p}{2})}}-\frac12\, \delta_{B(C_1,k_0+b),0}\right),
\end{split}
\ee
where $\Phi_1(x)=\tfrac12\, \bigl[\sign B(C_2,x)-\sign B(C'_2,x) \bigr]$ and $\delta_{i,j}$ is
the Kronecker delta function.

The proof of convergence proceeds as in \cite{Zwegers-thesis}. The second line in \eqref{indefresummed}
is bounded for all $k_0$ in the sum since $B(C_1',k_0+b)\neq 0$. It
remains to show that $Q(x)$ is negative definite for $x\in
\left< C_1\right>^\perp$ and over the range where $\Phi_1(x)$ is non-vanishing.
To this end, note that $Q$ has signature $(1,n-2)$ on $\left< C_1 \right>^\perp$.
As noted below \eqref{defK}, the projections of $C_2$ and $C'_2$  on
$\left< C_1 \right>^\perp$ satisfy $Q(C_{2\perp 1}),Q(C_{2'\perp 1}), B(C_{2\perp 1},C_{2'\perp 1})>0$,
therefore the same argument as in the  proof of Theorem \ref{modularLorentz} shows that
the sum over $k_0$ is absolutely convergent.

Note that the convergence of $\vartheta[\Phi_2]$ requires $B(C_1',k_0+b)\neq 0$. However,
setting $z=b\tau-c$, the  right-hand side of \eqref{indefresummed} can be analytically continued
to the larger domain $z\in\IC^n\backslash \cP_{C'_1,\mu}$ where $\cP_{C'_1,\mu}$ is the divisor
\be
\label{poleloc}
\cP_{C'_1,\mu} = \left\{ z \in \IC^{n} :\
\exists k\in\Lambda_1,\
B(C'_1, z+\tau\bfk - \tfrac{\bfp}{2}) \in \IZ\,\right\} .
\ee

Using \eqref{E2CCfromM2}, the additional terms in the modular completion
$\vartheta_\mu[\widehat \Phi_2]$ can be written as
\be
\label{prod2signaddl}
\begin{split}
\widehat\Phi_2-\Phi_2 = &\, \frac14\,  \Bigl\{  M_2(C_1,C_2;x) - M_2(C_1,C_2';x)
\Bigr.
+ \Bigl[ \sign B(C_{2\perp1},x) - \sign B(C_{2'\perp1},x) \Bigr]
M_1(C_1;x)
\\&\,
+ \Bigl[  \sign B(C_{1\perp2},x) - \sign B(C_{1'},x) \Bigr]
M_1(C_2;x)
\\&\,
\Bigl.
+\Bigl[ \sign B(C_{1'},x) - \sign B(C_{1\perp2'},x) \Bigr]
M_1(C'_2;x)
\Bigr\}.
\end{split}
\ee
Convergence follows from the conditions \eqref{condl} and inequalities below \eqref{defK}.
The shadow of $\vartheta_\mu[\widehat \Phi_2]$ is now the theta series
$\vartheta_{\bfmu}[\Psi_2]$ with kernel
\be
\label{prod2signshadowl}
\begin{split}
\Psi_2(x) =&\, \frac{\I}{8}\, \Biggl\{ \tfrac{B(C_1,x)}{\sqrt{Q(C_1)}}\, e^{-\frac{\pi B(C_1,x)^2}{Q(C_1)}}
\Bigl[E_1(C_{2\perp1};x)-E_1(C_{2'\perp1};x)\Bigr]
\Biggr.
\\
&\,
+ \tfrac{B(C_2,x)}{\sqrt{Q(C_2)}}\, e^{- \frac{\pi B(C_2,x)^2}{Q(C_2)}}
\Bigl[ E_1(C_{1\perp2};x)- \sign B(C_1',x)\Bigr]
\\
&\,\Biggl.
+ \tfrac{B(C'_2,x)}{\sqrt{Q(C'_2)}}\,
e^{- \frac{\pi B(C'_2,x)^2}{Q(C'_2)}}
\Bigl[ \sign B(C_1',x) -E_1(C_{1\perp2'},x)\Bigr]
\Biggr\}.
\end{split}
\ee

\subsection{Double null limit}
\label{doublenull}

Finally, we consider the case where both $C_1'$ and $C_2'$ degenerate to
linearly independent null vectors in $\Lambda$, with $B(C'_1,C_2)=B(C'_1,C'_2)=B(C_1,C'_2)=0$.
The vectors $C_1$ and $C_2$ remain time-like. Then the determinants and cofactors reduce to
\be
\begin{split}
\Delta_{11'22'}=&B(C_1,C'_1)^2\,B(C_2,C'_2)^2, \quad
D_{1'2'}= -B(C_1,C'_1) B(C_1,C_2) B(C_2,C'_2),
\\
D_{11'}=&\,B(C_1,C'_1) \,B(C_2,C'_2)^2, \quad
D_{22'}=B(C_1,C'_1)^2 B(C_2,C_2'),
\end{split}
\ee
and $\Delta_{1'22'}=\Delta_{11'2'}=D_{12}=D_{12'}=D_{1'2}=0$. The conditions \eqref{cond_thm1}
for the convergence of $\vartheta_\mu[\Phi_2]$ now become
\begin{subequations}
\label{condII}
\bea
\label{cond1lI}&Q(C_1),\,  Q(C_2)>0,\quad  Q(C'_1)=Q(C'_2)=0,
\\
\label{cond34lI}& B(C_1,C'_1)>0 , \quad  B(C_2,C'_2)>0 , 
\\
\label{cond5II}& {\operatorname M}_{00} \mbox{ is negative definite},
\eea
\end{subequations}
where ${\operatorname M}_{00}=\begin{pmatrix}
  \Delta_{122'}  & D_{1'2'}\\
 D_{1'2'}  & \Delta_{11'2} \end{pmatrix}$.
It follows from the positivity of $\det({\operatorname M}_{00})$ that $\Delta_{12}>0$.

One can again decompose $k\in \Lambda+\mu+\tfrac{\bfp}{2}$  into $k=k_0 + m_1 C'_1+m_2C'_2$
with $m_1,m_2\in \IZ$ and $k_0\in \Lambda_{12}$ where
\be
\Lambda_{12}=\left\{k_0\in \Lambda+\mu+\tfrac{\bfp}{2} \ :\
\tfrac{B(C_1,k_0+b)}{B(C_1,C'_1)}\in [0,1)\ {\rm and}\ \tfrac{B(C_2,k_0+b)}{B(C_2,C'_2)}\in [0,1)\right\}.
\label{latticeLCC}
\ee
The sum over $m_1$ and $m_2$ is convergent provided $B(C'_1,k_0+b)\neq 0$ and $B(C'_2,k_0+b)\neq 0$ for all $k_0\in \Lambda_{12}$.
Summing up the  double geometric series, one obtains
\be
\begin{split}
\label{thetaresummed2}
\vartheta_\mu[\Phi_2] &=  \sum_{\bfk_0\in \Lambda_{12}}
(-1)^{B(\bfk_0,\bfp)}\, q^{-\frac12 Q(\bfk_0+\bfb)}\,\expe{B(\bfc,\bfk_0+\haf\bfb)}
\\
&\times \left(\tfrac{1}{1-\expe{ B(C'_1, \bfc-\tau(\bfk_0+b)+\frac{p}{2})}}-\frac12\, \delta_{B(C_1,k_0+b),0}\right)
\left(\tfrac{1}{1-\expe{ B(C'_2, \bfc-\tau(\bfk_0+b)+\frac{p}{2})}\,}-\frac12\, \delta_{B(C_2,k_0+b),0}\right).
\end{split}
\ee
As in the discussion below \eqref{indefresummed}, the second line is bounded for all $k_0$ in the sum.
Furthermore, since $\Delta_{12}>0$, $Q(k)$ is negative definite on the orthogonal complement
$\left<C_1,C_2\right>^\perp=\left\{ x\in \IR^{2,n-2}\ :\  B(C_1,x)=B(C_2,x)=0\right\}$,
one concludes that $\vartheta[\Phi_2]$ is convergent.
Setting $z=b\tau-c$, the right-hand side of \eqref{thetaresummed2}
can be analytically continued  to a meromorphic function of $z\in \IC^n$
(in the sense of Remark \ref{rk_hol}), with simple poles on the divisors
$\cP_{C'_1,\mu}$ and $\cP_{C'_2,\mu}$.

The additional terms in the modular completion $\vartheta_\mu[\widehat\Phi_2]$ are given by
\be
\label{prod2signaddll}
\begin{split}
\widehat\Phi_2-\Phi_2 = &\, \frac14\,  \Bigl\{  M_2(C_1,C_2;x)
+ \Bigl[ \sign B(C_{2\perp1},x) - \sign B(C_{2'},x)  \Bigr]
M_1(C_1;x)
\Bigr.
\\
&\,\Bigl.
+ \Bigl[  \sign B(C_{1\perp2},x) - \sign B(C_{1'},x) \Bigr]
M_1(C_2;x)
\Bigr\}.
\end{split}
\ee
Convergence follows by the same arguments as before.
The shadow of $\vartheta_\mu[\widehat\Phi_2]$ is now the theta series
$\vartheta_{\bfmu}[\Psi_2]$ with
\be
\label{prod2signshadowll}
\begin{split}
\Psi_2(x) =&\, \frac{\I}{8}\, \Biggl\{ \tfrac{B(C_1,x)}{\sqrt{Q(C_1)}}\,e^{-\frac{\pi B(C_1,x)^2}{Q(C_1)}}
\Bigl[ E_1(C_{2\perp1};x) -  \sign B(C'_2,x) \Biggr]
\Biggr.
\\
&\,\Biggl.
+ \tfrac{B(C_2,x)}{Q(C_2)}\, e^{-\frac{\pi B(C_2,x)^2}{Q(C_2)}}
\Bigl[ E_1(C_{1\perp2};x)- \sign B(C_1',x) \Bigr]\Biggr\}.
\end{split}
\ee

In addition to $C'_1$ and $C'_2$, one may also let $C_1$ and $C_2$ degenerate to null
vectors in $\Lambda$, analogously to Remark \ref{remarknullvectors}
for Lorentzian lattices. In such cases, $\Psi_2$ vanishes and $\vartheta_\mu[\Phi_2]$ can
be analytically continued to a meromorphic Jacobi form.

\subsection{Signature $(2,1)$ and special case $C_1=C_2$ \label{sec_sig21}}

In the previous construction it was important that $\Delta_{11'22'}>0$. 
Unfortunately, this condition rules out the case $n=3$, where $\Delta_{11'22'}$ necessarily vanishes. 
It is therefore natural to ask if convergent conformal theta series can be constructed 
using only three vectors.\footnote{We thank Don Zagier for suggesting this possibility, Jeffrey Harvey for drawing
our attention to \cite{Zwegers:2009}, and Kathrin Bringmann and Larry Rolen for a discussion
about the theta series in \cite{BRZ2015}.}
Indeed, two of Ramanujan's fifth order mock theta functions, $\chi_0(q)$ and $\chi_1(q)$, are examples  
of signature $(2,1)$ theta series of the form \eqref{prod2sign}
upon multiplication by the Dedekind eta function \cite{Zwegers:2009}. The corresponding quadratic form $Q$ and the vectors $\{C_i,C'_i\}$ are given by
\be  
\label{Ramachi1}
Q=- \begin{pmatrix} 1 & 2 & 2 \\ 2 & 1 & 2 \\ 2 & 2 & 1 \end{pmatrix},\quad
C_1=C_2=\begin{pmatrix} 3 \\ -2 \\ -2 \end{pmatrix},\quad
C'_1=\begin{pmatrix} 2 \\ -3 \\ 2 \end{pmatrix},\quad
C'_2=\begin{pmatrix} 2 \\2 \\ -3 \end{pmatrix}.
\ee
Another example arises in the context of quantum invariants of torus knots:
after multiplication by $\eta^3/\theta_1$, the colored Jones polynomial for
the torus knot $T_{(2,2t+1)}$  is given
by a  signature $(2,1)$ theta series  of the form \eqref{prod2sign}
with \cite[Thm. 1.3]{zbMATH06587777}
\be
\label{HikamiUC}
Q=- \begin{pmatrix} \tfrac14 & \tfrac{4t+3}{4}  & \tfrac12 \\
\tfrac{4t+3}{4} & \tfrac14 & \tfrac12 \\
\tfrac12  & \tfrac12 & 0 \end{pmatrix}\!,\ 
C_1=C_2=\begin{pmatrix} 1 \\ -1 \\ -2t-1 \end{pmatrix}\!,\ 
C'_1=\begin{pmatrix} 1 \\ -1 \\ 2t+1 \end{pmatrix}\!,\ 
C'_2=\begin{pmatrix} 1 \\ 1 \\ -2(t+1) \end{pmatrix}\!.
\ee
Yet another example is the theta series $F(z_1,z_2,z_3;\tau)$ considered in \cite{BRZ2015} in
the context of open Gromov-Witten invariants, which is a signature $(2,1)$ theta series  
of the form \eqref{prod2sign} with
\be
\label{BRZFC}
Q=- \begin{pmatrix} 1 & 1 & 1 \\ 1 & 0 & 1 \\ 1 & 1 & 0 \end{pmatrix},\quad
C_1=C_2=\begin{pmatrix} 1 \\ -1 \\ -1 \end{pmatrix},\quad
C'_1=\begin{pmatrix} -1 \\ 1 \\ 0 \end{pmatrix},\quad
C'_2=\begin{pmatrix} -1 \\0 \\ 1 \end{pmatrix}.
\ee

Motivated by these examples, let us investigate the convergence of
conformal theta series  $\vartheta_{\bfmu}[\Phi_2,0]$ with $\Phi_2$ given by \eqref{prod2sign}
with $C_1=C_2$. For this purpose we write the Gram determinant $\Delta(x,C_1,C'_1,C'_2)$ as
\be
\begin{split}
\Delta(x,C_1,C'_1,C'_2)=&
Q_-(x) \Delta_{11'2'}
-  \Delta_{1'2'} \, [B(C_1,x)]^2 -  \Delta_{12'} [B(C'_1,x)]^2 - \Delta_{11'}\,  [B(C'_2,x)]^2,
\end{split}
   \ee
where
\be
Q_-(x)=Q(x)  + 2 \frac{E_{11'} \, B(C_1,x)\, B(C'_1,x) + E_{12'}\, B(C_1,x) \,B(C'_2,x)
+E_{1'2'}  \, B(C'_1,x)\, (C'_2,x)}{ \Delta_{11'2}}
\ee
and
\be
\begin{split}
E_{11'} =&
Q(C'_2) \, B (C_1,C'_1) - B (C'_1,C'_2)\, B(C_1,C'_2)
= Q(C'_2) \,  B(C_{1\perp2'}, C_{1'\perp2'}),
\\
E_{12'} =&
 Q(C'_1)\, B (C_1,C'_2)  - B (C'_1,C'_2)  \,  B (C_1,C'_1)
= Q(C'_1) \, B(C_{1\perp1'},C_{2'\perp1'}),
\\
E_{1'2'} =&
Q(C_1)\, B (C'_1,C'_2)  -  B (C_1,C'_1) \, B (C_1,C'_2)
=  Q(C_1)\, B(C_{2'\perp1}, C_{1'\perp1})\ .
\end{split}
\ee
We assume  that  $\langle C_1,C'_1,C'_2 \rangle$ span a signature $(2,1)$ plane,
so that $ \Delta_{11'2'} <0$, and moreover that
\be
\label{condDel0}
\Delta_{1'2'}, \Delta_{12'}, \Delta_{11'}> 0\ ,\quad
E_{11'}, E_{12'} \geq 0, \quad E_{1'2'}\leq 0 .
\ee
A similar argument as below \eqref{Gram50mat} shows that the quadratic form $Q_-(x)$ is
negative definite. Moreover, the locally constant
function \eqref{prod2sign} vanishes unless  $B(C_1,x) B(C'_1,x) \leq 0$ and
$B(C_1,x) B(C'_2,x) \leq 0$. These two conditions imply
$B(C'_1,x) B(C'_2,x) \geq 0$, unless $B(C_1,x)=0$. Assuming that $b$ is chosen such that
$B(C_1,k+b)$ never vanishes for $k\in \Lambda+\mu+\tfrac12 p$, one has $Q(x)\leq Q_-(x)$
for all terms contributing to the sum,
which ensures that  $\vartheta_{\bfmu}[\Phi_2,0]$ is convergent.
The modular completion is obtained by adding the same terms as
\eqref{prod2signadd}, except that the terms $M_2(C_1,C_2;x)$, $\sign B(C_{2\perp1},x)$,
$\sign B(C_{1\perp2},x)$ all vanish (see the discussion around \eqref{E2Delta0}).
The conditions  $E_{11'},E_{12'} \geq 0,\ E_{1'2'}\leq 0$ guarantee that the additional
terms in the modular completion are themselves convergent.

The conditions \eqref{condDel0} are indeed satisfied for the choice \eqref{Ramachi1}
of quadratic form $Q$ and vectors $C_1,C'_1,C'_2$  relevant for the fifth order mock theta functions, where
\be
\Delta_{11'2'}=-5^5, 
\quad 
\Delta_{1'2'}=\Delta_{11'}=\Delta_{12'}=5^3, 
\quad
E_{11'}=E_{12'}=-E_{1'2'}=250.
\ee
In contrast, the quadratic form and vectors \eqref{HikamiUC} relevant
for the quantum invariants of torus knots satisfy weaker conditions,
\be
\begin{split}
\Delta_{11'2'}=-4(2t+1)^5,
\quad
\Delta_{11'}=0,
\quad  
\Delta_{12'}=\Delta_{1'2'}=(2t+1)^3,
\\
E_{11'}=(2t+1)^3(4t+3), 
\quad 
E_{12'}=-E_{1'2'}=2(2t+1)^3,\quad
\end{split}
\ee
so $Q_-(x)$ has one null direction, which can   be treated similarly as in
\S\ref{sec_null}. Similarly, the quadratic form and vectors \eqref{BRZFC} relevant for  
the signature (2,1) theta series $F(z_1,z_2,z_3;\tau)$ satisfy 
\be
\Delta_{11'2'}=-1, 
\quad
\Delta_{11'}=\Delta_{12'}=0,
\quad
\Delta_{1'2'}=1,
\quad
 E_{11'}=E_{12'}=- E_{1'2'}=1,
\ee
so that $Q_-(x)$ has two null directions.  Since $B(C'_1,C'_2)=0$ and
$\Delta_{11'}=\Delta_{12'}=0$, the modular completion of $F(z_1,z_2,z_3;\tau)$
involves only $M_1$ functions,
in agreement with the fact that it can be written in terms of the
standard Appell-Lerch sum \cite{BRZ2015}.

More generally, one could consider theta series of the form \eqref{prod2sign}
where $C_2$ is a linear combination of $C_1$ and $C'_1$, although we expect
that this can always be reduced to the case $C_2=C_1$.

\section{Application: Generalized Appell-Lerch sums \label{sec_appell}}
\label{genappell}

In this section, we apply the techniques introduced in Sections \ref{derror}
and \ref{2n-2} to find the modular completion of the generalized Appell-Lerch sum $\mu_{2,2}(u,v)$, which will be
defined in Definition \ref{Defmus} and which is based on the ${\operatorname{A}_2}$ root lattice.
The subscript denotes here the signature
of the indefinite lattice obtained after expanding the denominators as a geometric series.
A specialization of this function appears in the generating function of invariants of moduli spaces of
rank 3 vector bundles on $\mathbb{P}^2$ as explained in some more detail in the Introduction.
We start with a few preliminary facts
on theta series associated to the ${\operatorname{A}_2}$ lattice, which will be useful for the following.

\subsection{Preliminaries}

Let $\tau\in \mathbb{H}$ and $v\in \mathbb{C}$. Let $\mathcal{Q}_r$ be
the \bp{opposite of the} quadratic form of the ${\operatorname{A}_r}$ root lattice.  For
$r=2$, we choose the basis of the lattice such that
$\bp{-}\cQ_2(k)=2k_1^2+2k_2^2+2k_1k_2$. The corresponding theta functions are known as ``cubic theta functions''
(see for example \cite{borwein1994}), and defined by
\be
\Theta_{{\cQ_2},j}(v,\tau):=\sum_{k_1,k_2\in \mathbb{Z}+\frac{j}{3}} e^{2\pi\I v(k_1+2k_2)}
q^{\bp{-}\frac{1}{2}\cQ_2(k)}.
\ee
They satisfy the quasi-periodicity property\footnote{We will often omit $\tau$ from the arguments
of a function in the following. For example, $\Theta_{{\cQ_2},j}(v,\tau)$ will be
referred to as $\Theta_{{\cQ_2},j}(v)$.}
\be
\Theta_{{\cQ_2},j}(v+\lambda \tau+\mu)=q^{-\lambda^2}e^{-4\pi\I \lambda v} \Theta_{{\cQ_2},j}(v).
\ee
Under the generators $S$ and $T$ of $\operatorname{SL}(2,\IZ)$, they transform as a vector valued modular form
\be
\begin{split}
&S:\qquad \Theta_{{\cQ_2},j}\!\left(\frac{v}{\tau},-\frac{1}{\tau}\right)=
-\I \tau e^{2\pi\I v^2/\tau}\sum_{\ell \!\!\!\mod 3} e^{-2\pi \I j\ell/3} \,\Theta_{{\cQ_2},\ell} (v,\tau),
\\
&T:\qquad \Theta_{{\cQ_2},j}(v,\tau+1)=e^{2\pi\I j^2/3}\,\Theta_{{\cQ_2},j}(v,\tau).
\end{split}
\ee

Since $\Theta_{{\cQ_2},0}(v)$ will occur repeatedly in the following, we denote it for brevity by $\Theta(v)$. It
is a Jacobi form of weight 1 and index 1 for the congruence subgroup $\Gamma_0(3)$
with multiplier $\left(\frac{d}{3}\right)$ (the Legendre symbol modulo 3) \cite{borwein1994}.
The zeroes of $\Theta(v)$ are at $v_{\pm}=\pm \left(\frac{1}{2}\tau+\frac{1}{4}+\nu(\tau) \right)$
 modulo $\tau\mathbb{Z}+\mathbb{Z}$, where
$\nu(\tau)=-2\pi\int_{\tau}^{\I\infty}(t-\tau)\,F(t)\,\de t$, with $F(\tau)$ the modular form \cite{Manschot:2014cca}
\be
F(\tau):=\frac{\eta(\tau)^{12}\,\eta(3\tau)^6}{\left(q^{\frac{1}{2}}
\prod_{\lambda,\mu=0,1}\Theta(\tfrac{\lambda\tau+\mu}{4},\tau)\right)^{3/2}}\, .
\ee

\subsection{Generalized Appell-Lerch sums}

We define the following generalized Appell-Lerch sums of signature $(2,1)$ and $(2,2)$.

\begin{definition}
\label{Defmus}
Let $u\in \mathbb{C}\backslash (\tau\mathbb{Z}+\mathbb{Z})$ and $v\in
\mathbb{C}\backslash (v_{\pm}+\tau\mathbb{Z}+\mathbb{Z})$.
Then we define $\mu_{2,1}(u,v)$ and $\mu_{2,2}(u,v)$ by
\be
\begin{split}
\label{defmus}
\mu_{2,1}(u,v,\tau)&:=\frac{1}{2}+\frac{e^{2\pi \I u}}{\Theta(v)}\sum_{k_1,k_2\in \mathbb{Z}}
\frac{e^{2\pi \I v (k_1+2k_2)}q^{k_1^2+k_2^2+k_1k_2+2k_1+k_2}}{1-e^{2\pi \I u}q^{2k_1+k_2}}\, ,
\\
\mu_{2,2}(u,v,\tau)&:=\frac{1}{4}+\frac{e^{2\pi \I u}}{\Theta(v)}\sum_{k_1,k_2\in \mathbb{Z}}
\frac{e^{2\pi\I v(k_1+2k_2)}q^{k_1^2+k_2^2+k_1k_2+2k_1+k_2}}{(1-e^{2\pi\I u}q^{2k_1+k_2})(1-e^{2\pi\I u}q^{k_2-k_1})}\, .
\end{split}
\ee
\end{definition}
For simplicity, we included only two elliptic variables $u$ and $v$ in the definition
of $\mu_{2,1}(u,v)$ and $\mu_{2,2}(u,v)$. It is straightforward to
refine the functions with one or two more elliptic
variables. The functions are specializations of $\mu_{\cQ_2, \{m_j\}}({\boldsymbol u},{\boldsymbol v})$
(\ref{genAppell}) up to a constant term. For example, for $\mu_{2,2}(u,v)$,
one sets $m_1=(1,0)$, $m_2=(-1,1)$, ${\boldsymbol u}=(u,u)$ and ${\boldsymbol v}=(\tau,v)$.

The modular completion of $\mu_{2,1}(u,v)$ is readily determined
after expressing $\mu_{2,1}(u,v)$ in terms of the
classical Appell-Lerch sum $\mu(u,v)$ (\ref{ALsum}) and Jacobi theta
functions. To state the modular completion $\widehat \mu_{2,2}(u,v)$
of $\mu_{2,2}(u,v)$ and prove its modular properties in Theorem
\ref{thcompmu22}, we make the following

\begin{definition}
\label{DefRs}
Let $u$ and $v$ be in the same domain as in Definition \ref{Defmus}
with $a=\Im(u)/\tau_2$. Then we define $R_{2,2}(u,v)$ by
\begin{eqnarray}
R_{2,2}(u,v,\tau)&:=&R_2(u-v)
\non \\
&& +4 \mu_{2,1}(u,v)\,R_{1,0}(u-v)
+4\,e^{\pi \I (v-u)}q^{-\frac{1}{4}}\left( \mu_{2,1}(u,v-\tau)-\textstyle{\frac{1}{2}}\right)\, R_{1,1}(u-v),
\non
\end{eqnarray}
where $\mu_{2,1}(u,v)$ is as in Definition \ref{Defmus}. Moreover, $R_{1,j}(u)$ with $j=0,1$, is defined by
\begin{eqnarray}
\label{compR}
R_{1,j}(u,\tau)&:=&
\sum_{\ell\in \mathbb{Z}+\frac{j}{2}}
\Bigl[\sign(\ell)-E_1\left(\sqrt{\tau_2}\left(2\ell-a\right)\right)\Bigr]
e^{3\pi \I \ell u}q^{-3\ell^2},
\end{eqnarray}
with $E_1(u)$ as in Theorem \ref{modularLorentz}, and $R_{2}(u)$ is defined by
\begin{eqnarray}
\label{compR2}
R_{2}(u,\tau)&:=&
\sum_{k_3,k_4\in \mathbb{Z}} \biggl\{ \sign(k_3)\sign(k_4)
-2 \Bigl[ \sign(k_4) -E_1(\sqrt{3\tau_2}(k_4-a))\Bigr]\sign(2k_3-k_4) \biggr.
\non
\\
&&\biggl. -E_2\!\left(\tfrac{1}{\sqrt{3}}\,;\sqrt{\tau_2} (2k_3-k_4-a), \sqrt{3\tau_2} (k_4-a)\right)
\biggr\}\,  e^{2\pi\I(k_3+k_4)u}q^{-k_3^2-k_4^2+k_3k_4}
\end{eqnarray}
with $E_2(\alpha,u_1,u_2)$ as in Definition \ref{DefE2}.
\end{definition}

\begin{theorem}
\label{thcompmu22}
With $\mu_{2,2}(u,v)$ as in (\ref{defmus}), we define $\widehat \mu_{2,2}(u,v)$ by
\be
\widehat \mu_{2,2}(u,v,\tau):=\mu_{2,2}(u,v,\tau)-\textstyle{\frac{1}{4}}\, R_{2,2}(u,v,\tau).
\ee
Then $\widehat \mu_{2,2}(u,v)$ transforms as a two-variable Jacobi form of weight 1 under
the modular group $\Gamma_0(3)$ with multiplier $\left(\frac{d}{3}\right)$, and matrix-valued index
$m={\scriptsize \left(\begin{array}{cc}\! -1\! & 1\\ 1 & \!-1\! \end{array}\right)}$.
\end{theorem}

\begin{proof}
The theorem is proven by first expressing $\mu_{2,2}(u,v)$ in terms of
an indefinite theta function for signature $(2,2)$ with both $C'_1$ and $C'_2$
null as in Section \ref{doublenull}, and then using the results of Section \ref{conf_theta}
to determine its modular completion. To this end, we start by defining
the generalized Appell function\footnote{We use the terminology
``generalized Appell function'' for functions which can be written as a generalized Appell-Lerch function
(\ref{genAppell}) multiplied by $\Theta_{\cQ}(v)$.}
\begin{eqnarray}
A_{2,2}(u,v,\tau)& :=&e^{2\pi \I u}\sum_{k_1,k_2\in \mathbb{Z}}
\frac{e^{2\pi \I v (k_1+2k_2)}q^{k_1^2+k_2^2+k_1k_2+2k_1+k_2}}{(1-e^{2\pi \I u}q^{2k_1+k_2}) (1-e^{2\pi \I u}q^{k_2-k_1})}
\\
&=&\Theta(v)\,\left(\mu_{2,2}(u,v)-\textstyle{\frac{1}{4}}\right),
\non
\end{eqnarray}
and expand the denominator as a double geometric series,
\begin{eqnarray}
\label{A22exp}
 A_{2,2}(u,v)&=&{\frac{1}{4}}\sum_{k\in \mathbb{Z}^4}
 \left[\sign(k_3+\epsilon)+\sign(2k_1+k_2+a)\right]
 \left[\sign(k_4+\epsilon)+\sign(k_2-k_1+a) \right]
 \non \\
 && \times e^{2\pi\I (v (k_1+2k_2)+u (k_3+k_4+1))}q^{k_1^2+k_2^2+k_1k_2+(2k_1+k_2)(k_3+1)+(k_2-k_1)k_4},
\end{eqnarray}
where $k=(k_1,k_2,k_3,k_4)\in \mathbb{Z}^4$, $0<\epsilon<1$ and $a=\Im(u)/\tau_2$ as before. In the notation of Section \ref{sec_vz}, this corresponds to an indefinite theta series of the form
\eqref{Vignerasth} for the  signature $(2,2)$  bilinear form $B(x,y)$ with matrix representation
\be
{\rm B}=-\left(\begin{array}{cccc} 2 & 1 & 2 & -1 \\ 1 & 2 & 1 & 1\\ 2 & 1 & 0 & 0 \\ -1 & 1 & 0 & 0 \end{array} \right)\ ,
\ee
 Indeed,
shifting $k_3\mapsto k_3-1$  in (\ref{A22exp}), the second line is recognized
 as $e^{-2\pi\I B(k,z)}q^{-\tfrac12Q(k)}$
with $Q(k)=B(k,k)$ and $z=(0,u,v-u,v-u)$. Denoting  $b=\Im(v)/\tau_2$,  $A_{2,2}(u,v)$ can
be rewritten as
\begin{eqnarray}
\label{indefA22}
A_{2,2}(u,v)&=&{\frac{1}{4}}\sum_{k\in\mathbb{Z}^4}
\Bigl[\sign\left(k_3+b-a\right)+\sign\left(2k_1+k_2+a\right) \Bigr]
\non \\
&&\times \Bigl[\sign\left(k_4+b-a\right)+\sign\left(k_2-k_1+a\right)\Bigr]
\,e^{-2\pi\I B(k,z)}q^{-Q(k)/2}
\\
&&+{\frac{1}{4}}\sum_{k\in\mathbb{Z}^4} s(k,z,\epsilon)\, e^{-2\pi\I B(k,z)}
q^{-\tfrac12Q(k)}
\non
\end{eqnarray}
with $s(k,z,\epsilon)$ equal to the following combination of signs:
\begin{eqnarray}
s(k,z,\epsilon)&=&\Bigl[\sign(k_3-1+\epsilon)-\sign(k_3+b-a)\Bigr] \sign(k_2-k_1+a)
\non \\
&&+\Bigl[\sign(k_4+\epsilon)-\sign(k_4+b-a)\Bigr] \sign(2k_1+k_2+a)
\\
&&+\sign(k_3-1+\epsilon) \sign(k_4+\epsilon)
-\sign(k_3+b-a)\sign(k_4+b-a).
\non
\end{eqnarray}
The first two lines of (\ref{indefA22}) have precisely the form of an indefinite theta function
$\vartheta_{\bfmu}[\Phi_2]$ studied in Section \ref{2n-2}, with vanishing characteristic and glue
vectors $p=\mu=0$ and with
\be
C_1=\frac{1}{3}\left( \begin{array}{c} -1 \\ -1 \\ 2 \\ 1\end{array} \right),
\quad
C'_1=\left( \begin{array}{c} 0 \\ 0 \\ 1 \\ 0\end{array} \right),
\quad
C_2=\frac{1}{3}\left( \begin{array}{c} 1 \\- 2 \\ 1 \\ 2\end{array} \right),
\quad
C'_2=\left( \begin{array}{c} 0 \\ 0 \\ 0 \\ 1\end{array} \right).
\ee
These vectors have norms $Q(C_1)=Q(C_2)=\frac{2}{3}$, $Q(C'_1)=Q(C'_2)=0$,
and satisfy the properties listed in \eqref{condII}.
Following the discussion of the previous section, the completion of $A_{2,2}(u,v)$ is
obtained by appropriately replacing the products of sign functions on the first two lines
of (\ref{indefA22}) and subtracting the third line. The replacements are as follows
\be
\begin{split}
\sign(k_3+b-a) \,\sign(k_4+b-a) &\ \to\
E_2\!\left(\tfrac{1}{\sqrt{3}}\,;\sqrt{\tau_2} (2k_3-k_4+b-a), \sqrt{3\tau_2} (k_4+b-a)\right),
\\
\sign(k_3+b-a)\, \sign(k_2-k_1+a) &\ \to\
\sign(k_2-k_1+a)\, E_1\!\left( \sqrt{3\tau_2} (k_3+b-a)\right),
\\
\sign(2k_1+k_2+a) \,\sign(k_4+b-a) &\ \to\
 \sign(2k_1+k_2+a)\, E_1\!\left( \sqrt{3\tau_2} (k_4+b-a)\right).
\end{split}
\ee
Using the above prescription, we define the completion of $A_{2,2}(u,v,\tau)$ as
\be
\widehat A_{2,2}(u,v,\tau)=A_{2,2}(u,v)-\textstyle{\frac{1}{4}} \tilde R(u,v,\tau)
\ee
with
\begin{eqnarray}
\label{tR}
\tilde R(u,v,\tau)&=&\sum_{k\in \mathbb{Z}^4}
\biggl\{ \Bigl[\sign(k_3-1+\epsilon) -E_1(\sqrt{3\tau_2}(k_3+b-a)) \Bigr] \sign(k_2-k_1+a) \biggr.
\non\\
&&+\Bigl[\sign(k_4+\epsilon) -E_1(\sqrt{3\tau_2}(k_4+b-a))\Bigr] \sign(2k_1+k_2+a)
\\
&&
+\sign(k_3-1+\epsilon)\sign(k_4+\epsilon)
\non \\
&&\biggl.
-E_2\!\left(\tfrac{1}{\sqrt{3}}\,;\sqrt{\tau_2} (2k_3-k_4+b-a), \sqrt{3\tau_2} (k_4+b-a)\right)\biggr\}
\, e^{-2\pi\I B(k,z)}q^{-\tfrac12Q(k)}.
\non
\end{eqnarray}

The theorem will follow if we prove that $\tilde
R(u,v)=\Theta(v)\,(R_{2,2}(u,v)-1)$ with $R_{2,2}(u,v)$ as in
Definition \ref{DefRs}. To this end, note that the transformation
\be
\label{A22sym}
k_1\mapsto -k_1, \quad k_2\mapsto k_2+k_1, \quad k_3\mapsto k_4, \quad k_4\mapsto k_3,
\ee
is a symmetry of $B(k,z)$ and $Q(k)$. Using this symmetry, we rewrite (\ref{tR}) as
\begin{eqnarray}
\label{tR2}
&&\sum_{k\in \mathbb{Z}^4}\biggl[ 2\Bigl[\sign(k_4) -E_1(\sqrt{3\tau_2}(k_4+b-a))\Bigr]\sign(2k_1+k_2+a)
 \non \\
&& +\sign(k_3)\, \sign(k_4) -E_2\!\left(\tfrac{1}{\sqrt{3}}\,;\sqrt{\tau_2} (2k_3-k_4+b-a), \sqrt{3\tau_2} (k_4+b-a)\right)
\\
&&+s(k,(0,a,0,0),\epsilon)\biggr] \, e^{-2\pi\I B(k,z)}
q^{-\tfrac12Q(k)}.
\non
\end{eqnarray}

We will now show that the contribution of the third line in (\ref{tR2}) equals $-\Theta(v)$.
We distinguish four cases, depending on whether $k_3=0$ or $\neq 0$, $k_4=0$ or $\neq 0$.
Using the symmetry (\ref{A22sym}), one verifies easily that the contribution for $k_3=k_4=0$ equals
\be
\label{cont1}
-\sum_{k\in \mathbb{Z}^4\atop k_3=k_4=0} e^{-2\pi\I B(k,z)}q^{-\tfrac12Q(k)}=-\Theta(v).
\ee
When $k_3=0$ and $k_4\neq 0$, the contribution is
\be
\label{cont2}
-\sum_{k\in \mathbb{Z}^4\atop k_3=0,\,k_4\neq 0} \Bigl[\sign(k_4+\epsilon)+\sign(k_2-k_1+a)\Bigr]\, e^{-2\pi\I B(k,z)}q^{-\tfrac12 Q(k)},
\ee
and similarly when $k_3\neq 0$ and $k_4=0$, the contribution is
\be
\label{cont3}
\sum_{k\in \mathbb{Z}^4\atop k_3\neq 0,\,k_4=0} \Bigl[\sign(k_3-1+\epsilon)+\sign(2k_1+k_2+a)\Bigr]\,
e^{-2\pi\I B(k,z)}q^{-\tfrac12Q(k)},
\ee
and when both $k_3,k_4\neq 0$, the contribution is zero.
Using again the symmetry (\ref{A22sym}) one can show that the sum of (\ref{cont2}) and (\ref{cont3}) vanishes.

Next we turn to the first line of (\ref{tR2}) and show that it naturally
factors in $\Theta(v)$ and other terms involving $\mu_{2,1}(u,v)$ and
$R_{1,j}(u,v)$ defined respectively in Definitions \ref{Defmus} and \ref{DefRs}.
We replace $\sign(2k_1+k_2+a)$ by $\sign(2k_1+k_2+a)+\sign(2k_3-k_4)-\sign(2k_3-k_4)$.
The terms multiplied by $\sign(2k_1+k_2+a)+\sign(2k_3-k_4)$ can be
resummed to the Appell function $\Theta(v)\,\mu_{2,1}(u,v)$.
To this end, we make the following shifts
\be
k_1\mapsto k_1+\tfrac{1}{2}k_4,\quad k_2\mapsto k_2-k_4,\quad k_3\mapsto k_3,\quad k_4\mapsto k_4,
\ee
such that these terms become
\be
\begin{split}
&2\sum_{k_2,k_3,k_4\in\mathbb{Z} \atop k_1+\frac{1}{2}k_4\in \mathbb{Z}}
\Bigl[\sign(2k_1+k_2+a)+\sign(2k_3-k_4)\Bigr]\,\Bigl[\sign(k_4)-E_1(\sqrt{3\tau_2}(k_4+b-a))\Bigr]
\\
&\times
e^{2\pi\I v(k_1+2k_2-\frac{3}{2}k_4)+2\pi\I u (k_3+k_4)}
 q^{k_1^2+k_2^2+k_1k_2+(2k_1+k_2)(k_3-\frac{1}{2}k_4)-\frac{3}{2}k_4^2}.
\end{split}
\ee
The sum over $k_3$ can now be carried out as a geometric sum. Distinguishing between $k_4$ even or odd,
the sum over $k_1$ and $k_2$ is proportional to either $\Theta(v)\,\mu_{2,1}(u,v)$ or $\Theta(v)\,(\mu_{2,1}(u,v-\tau)-\tfrac12)$.
Finally, use in the remaining term of the first line in (\ref{tR2}), i.e. the term multiplying $-\sign(2k_3-k_4)$,
the transformation
\be
\label{A22trafo}
k_1\mapsto k_1-k_3+k_4,\quad k_2\mapsto k_2-k_4, \quad k_3 \mapsto k_3,\quad k_4\mapsto k_4.
\ee
Then the sum over $k_1$ and $k_2$ can be split off as the theta function $\Theta(v)$.
We then obtain that the first line in (\ref{tR2}) equals
\begin{eqnarray}
&&4\Theta(v)\,\mu_{2,1}(u,v)\,R_{1,0}(u-v)
+4\Theta(v)\,e^{\pi \I (v-u)}q^{-\frac{1}{4}} \left(\mu_{2,1}(u,v-\tau) -\textstyle{\frac{1}{2}}\right) R_{1,1}(u-v)
\non \\
&&-2\Theta(v)\sum_{k_3,k_4\in \mathbb{Z}}
\Bigl[ \sign k_4 -E_1(\sqrt{3\tau_2}(k_4+b-a)) \Bigr]\sign(2k_3-k_4) \,  e^{2\pi\I(k_3+k_4)(u-v)}
q^{-k_3^2-k_4^2+k_3k_4}.
\non\\
\end{eqnarray}

Combining this result with the second line in (\ref{tR2}), we arrive at
\begin{eqnarray}
\tilde R(u,v)&=&4\Theta(v) \mu_{2,1}(u,v)\,R_{1,0}(u-v)
\non \\&&
+4\Theta(v)\,e^{\pi \I (v-u)}q^{-\frac{1}{4}} \left(\mu_{2,1}(u,v-\tau)-\textstyle{\frac{1}{2}}\right) R_{1,1}(u-v)
 \\
&&+\Theta(v)\,R_2(u-v)-\Theta(v)
\non
\end{eqnarray}
with $R_2(u)$ defined in (\ref{compR2}). After dividing by $\Theta(v)$ we arrive at the desired result for $R_{2,2}(u,v)$.

To discuss the modular properties of $\widehat A_{2,2}(u,v)$ we bring $\operatorname B$ into a block diagonal form using
the transformation (\ref{A22trafo}). The two blocks correspond each to a copy of the ${\rm A}_2$ root lattice,
\be
-\left(\begin{array}{cccc} 2 & 1 & 0 & 0 \\ 1 & 2 & 0 & 0\\ 0 & 0 & -2 & 1 \\ 0 & 0 & 1 & -2 \end{array} \right).
\ee
According to the general results of Section \ref{2n-2}, the modular properties of $\widehat A_{2,2}(u,v)$
are completely determined by this quadratic form and $\widehat \Phi_2$.
In particular, $\widehat A_{2,2}(u,v)$  has weight 2 and transforms under
the same congruence subgroup $\Gamma_0(3)$ as $\Theta(v)$, but with multiplier $\left(\frac{d}{3}\right)^2=1$.
Therefore, for $\scriptsize\left(\begin{array}{cc}\! a\! &\! b \!\\ \!c\! & \!d\! \end{array}\right)\in \Gamma_0(3)$
we have the following transformation for $\widehat A_{2,2}(u,v)$
\be
\widehat A_{2,2}\!\left(\frac{u}{c\tau+d},\frac{v}{c\tau+d},\frac{a\tau+b}{c\tau+d}\right)
=(c\tau+d)^2\, e^{-\pi \I \frac{c\,Q(z)}{c\tau+d}} \widehat A_{2,2}(u,v,\tau),
\ee
where $z=(0,u,v-u,v-u)$ as before. As a result, we obtain that $\widehat \mu_{2,2}(u,v)$
transforms as a two-variable Jacobi form of weight 1 under $\Gamma_0(3)$ with multiplier
$\left(\frac{d}{3}\right)$. The two-variable index follows from $-v^2-\tfrac12 Q(z)=-(u-v)^2$,
where $-v^2$ arises due to the division by $\Theta(v)$.
The matrix-valued index is therefore equal to $\scriptsize\left(\begin{array}{cc} \!-1\! & \!1\! \\ \!1\! & \!-1\!  \end{array} \right)$.

\end{proof}

\section{Extension to signature $(n_+,n_-)$ with $n_+>2$ \label{sec_gen}}

\subsection{General procedure}
The generalization of our construction to signature $(n_+,n_-)$ with $n_+>2$ is conceptually
straightforward. In analogy to \eqref{defE1int} and \eqref{defE2Cint}, for any set of linearly
independent  vectors $C_1,\dots C_{n_+}$ spanning a time-like $n_+$-dimensional
plane in $\IR^{n_+,n_-}$, we define the generalized error function $E_{n_+}(\{C_i\};x)$ on
$\IR^{n_+,n_-}$  by
\be
\label{defEn}
E_{n_+}(\{C_i\};x) := \int_{\left<C_1,\dots, C_{n_+}\right>}
\, \sign B(C_1,y) \cdots \sign B(C_{n_+},y) \,
e^{-\pi Q (y-x_+)}\, \de^{n_+} y\ ,
\ee
where $\de^{n_+} y$ is the uniform measure on the plane $\left<C_1,\dots, C_{n_+}\right>$ which is
normalized such that
$\int_{\left<C_1,\dots, C_{n_+}\right>} e^{-\pi Q(y)} \,  \de^{n_+} y=1$, and
$x_+$ is the orthogonal projection of $x$ on the same plane.
It is easy to
see\footnote{More generally, substituting the product of signs in \eqref{defEn} by
a locally polynomial function of $y$, homogeneous of degree $m$,  would
lead to a $C^\infty$  solution of Vign\'eras' equation with $\lambda=m$.}
 that \eqref{defEn} is a $C^\infty$  solution of Vign\'eras'
equation on $\IR^{n_+,n_-}$, which asymptotes to the locally constant function
$\prod_{i=1}^{n_+} \sign B(C_i,x)$ with exponential accuracy when all $|B(C_i,x)|\to\infty$.
More generally, in the asymptotic region where
$B(C_{n_1},x), \dots, B(C_{n_r},x)$ are kept finite while $|B(C_j,x)|\to\infty$ whenever
$j\neq \{n_1,\dots, n_r\}$, $E_{n_+}(\{C_i\};x)$ behaves as
\be
\label{Enlim}
E_{n_+}(\{C_i\};x) \sim E_r(\{C_{n_1},\dots, C_{n_r}\};x)
\prod_{j\neq \{n_1,\dots, n_r\}} \sign B(C_j,x)  ,
\ee
where $E_r$ is a generalized error function of lower order.

Similarly, in analogy with \eqref{defM1int} and \eqref{defM2Cint}, we define the complementary error function
$M_{n_+}(\{C_i\};x)$ by
\be
\label{defMn}
M_{n_+}(\{C_i\};x) := \left(\frac{\I}{\pi}\right)^{n_+}
\int_{\left<C_1,\dots, C_{n_+}\right>-\I x_+}  \frac{\sqrt{\Delta(\{C_i^\star\})} }
{B(C_1^\star,z)\cdots B(C_{n_+}^\star,z)}\,
e^{-\pi Q(z)-2\pi\I B(z,x)} \, \de^{n_+} z,
\ee
where $C_i^\star$ is the dual basis to $C_i$ in the  plane $\left<C_1,\dots, C_{n_+}\right>$ and
$\de^{n_+} z$ is the uniform measure on the same plane, normalized as
indicated below \eqref{defEn}.
It is easy to see\footnote{More generally, substituting the product of $1/B(C^\star_i,z)$ in
\eqref{defMn} by a rational function of $z$, homogeneous of degree $-m$,  would
lead to a solution of Vign\'eras' equation with $\lambda=m-n_+$.}
that $M_{n_+}(\{C_i\};x)$ is a solution of
Vign\'eras' equation with $\lambda=0$, $C^\infty$ away from the hyperplanes
$B(C_i^\star,x)=0$, exponentially \bp{decreasing}  away from the origin.

By Fourier transforming \eqref{defEn} over $y$ using the results of \cite{zbMATH01375120}
and deforming the contour of integration, one can in
principle relate $E_{n_+}(\{C_i\};x)$ to $M_{n_+}(\{C_i\};x)$. While we have  not carried out
this computation in detail,
evidence from the $n_+=1,2$ cases discussed above and $n_+=3$ discussed
below suggests that
\be
\label{EnMn}
E_{n_+}(\{C_i\};x)  = M_{n_+}(\{C_i\};x)+\cdots +
\prod_{i=1}^{n_+} \sign B(C_i,x),
\ee
where the dots denote linear combinations of $M_{r}$, $1\leq r<n_+$ with locally constant coefficients.
This decomposition generalizes \eqref{E2CCfromM2} to $n_+>2$.

Now we can use the functions $E_{n_+}$ to construct modular completions of indefinite theta series of signature $(n_+,n_-)$
as follows: for suitable choices of $n_+$-tuples $\{C_i\}$ and $\{C'_i\}$, which we shall not attempt
to characterize here, the locally constant function
\be
\label{prod2signg}
\Phi_{n_+}(x) :=
\frac1{2^{n_+}} \prod_{i=1}^{n_+} \Bigl[\sign B(C_i,x)-\sign B(C'_i,x) \Bigr]
\ee
defines a convergent holomorphic theta series $\vartheta_{\bfmu}[\Phi_{n_+},0]$,
whose modular completion is obtained by expanding out the product in \eqref{prod2signg} and replacing
each product of $n_+$ signs by the corresponding $E_{n_+}$ function.

\subsection{Triple error functions}

To illustrate this procedure, we now construct the special functions $E_3$ and
$M_3$ in some detail.
In a basis where the quadratic form $Q=x_1^2+x_2^2+x_3^2-\sum_{i=4\dots n} x_i^2$ is diagonal,
any three unit vectors $C_1,C_2,C_3$ spanning a time-like three-plane in $\IR^{3,n-3}$ can be rotated
via a $O(3,n-3)$ transformation
into
\be
\label{C123}
C_1 =\tfrac{1}{\sqrt{1+\alpha_2^2(1+\alpha_1^2)}}\begin{pmatrix}1\\\alpha_1 \alpha_2\\\alpha_2
\\ 0\\ \vdots\end{pmatrix},
\quad
C_2 =\tfrac{1}{\sqrt{1+\alpha_3^2(1+\alpha_2^2)}}\begin{pmatrix}\alpha_3 \\ 1\\ \alpha_2 \alpha_3
\\ 0\\ \vdots\end{pmatrix},
\quad
C_3=\tfrac{1}{\sqrt{1+\alpha_1^2(1+\alpha_3^2)}}\begin{pmatrix}\alpha_1 \alpha_3\\ \alpha_1\\ 1\\ 0\\ \vdots\end{pmatrix},
\ee
where $\alpha_1,\alpha_2,\alpha_3\in\IR$, $\alpha_1\alpha_2\alpha_3\neq 1$. Their
normalized dual basis is then
\be
\Cdtp=\tfrac{1}{\sqrt{1+\alpha_3^2}}\begin{pmatrix} 1 \\ -\alpha_3 \\ 0\\ 0\\ \vdots\end{pmatrix},
\qquad
\Ctup=\tfrac{1}{\sqrt{1+\alpha_1^2}}\begin{pmatrix} 0 \\ 1 \\ -\alpha_1\\ 0\\ \vdots\end{pmatrix},
\qquad
\Cudp=\tfrac{1}{\sqrt{1+\alpha_2^2}}\begin{pmatrix}-\alpha_2 \\ 0 \\1\\ 0\\ \vdots\end{pmatrix},
\ee
so that $\Cdtp$ points along the orthogonal projection of $C_1$ on the plane orthogonal to
$C_2$ and $C_3$.
For this choice of vectors, setting $u_i=x_i$ for $i=1,2,3$, the definitions \eqref{defEn} and \eqref{defMn}
specialize to
\be
\label{defE3}
\begin{split}
E_3(\alpha_i;u_i) =&   \int_{\IR^3}  e^{-\pi[ (u_1-y_1)^2+(u_2-y_2)^2+(u_3-y_3)^2]} \, 
\sign(y_1+\alpha_1\alpha_2 y_2+\alpha_2 y_3) 
\\  \times &
\sign(y_2+\alpha_2\alpha_3 y_3+\alpha_3 y_1) \,
\sign(y_3+\alpha_3\alpha_1 y_1+\alpha_1 y_2) \,\de y_1\, \de y_2 \,\de y_3 \, 
\end{split}
\ee
and
\be
\label{defM3}
M_3(\alpha_i;u_i)=\left(\frac{\I}{\pi}\right)^3 \int \frac{ |1-\alpha_1\alpha_2\alpha_3|\,
e^{-\pi (z_1^2+ z_2^2+z_3^2) -2\pi \I (u_1 z_1 + u_2 z_2+u_3 z_3)}
}{(z_1-\alpha_3 z_2)(z_2-\alpha_1 z_3)(z_3-\alpha_2 z_1)}\,
\, \de z_1\, \de z_2 \, \de z_3\,,
\ee
where $z_i$ is integrated over $\IR-\I  u_i$.
$E_3(\alpha_i;u_i) $ is a $C^\infty$ function
on $\IR^3$, while $M_3(\alpha_i;u_i) $ is a $C^\infty$ function on $\IR^3$ away from the loci
$u_1=\alpha_3 u_2$, $u_2=\alpha_1 u_3$, $u_3=\alpha_2 u_1$. When $\alpha_i=0$, so that
 the three vectors $C_i$ form an orthogonal basis,  $E_3$ and $M_3$ factorize into
 the product of three error functions,
\be
\label{M3E3zero}
E_3(0,0,0;u_i)  = E_1(u_1)\, E_1(u_2)\, E_1(u_3), \qquad
M_3(0,0,0;u_i)  = M_1(u_1)\, M_1(u_2)\, M_1(u_3).
\ee
Similarly, when $\alpha_2=\alpha_3=0$ so that $C_1$ is orthogonal to $C_2$ and $C_3$, $E_3$ and $M_3$
reduce to products of error and double error functions,
\be
\label{M3E3zero2}
E_3(\alpha_1,0,0;u_i)=E_1(u_1)\, E_2(\alpha_1; u_3,u_2),
\qquad
M_3(\alpha_1,0,0;u_i)=M_1(u_1)\, M_2(\alpha_1; u_3,u_2).
\ee
For arbitrary values of the $\alpha_i$'s, $E_3$ and $M_3$ are easily checked to
be solutions of Vign\'eras' equation  with $\lambda=0$ on $\IR^3$ and
$\IR^3 \backslash \{ (u_1-\alpha_3 u_2)(u_2-\alpha_1 u_3)(u_3-\alpha_2 u_1)=0\}$,
respectively. For $|u|\to\infty$ along
a fixed ray,
\be
\label{EM3largex}
\begin{split}
E_3(\alpha_i;u_i)\sim &\,
\sign(u_1+\alpha_1\alpha_2 u_2+\alpha_2 u_3) \,
\sign(u_2+\alpha_2\alpha_3 u_3+\alpha_3 u_1) \,
\sign(u_3+\alpha_3\alpha_1 u_1+\alpha_1 u_2),
\\
M_3(\alpha_i;u_i) \sim &\, - \frac{|1-\alpha_1\alpha_2\alpha_3|\, e^{-\pi(u_1^2+u_2^2+u_3^2)}}
{\pi^3 \, (u_1-\alpha_3 u_2)\, (u_2-\alpha_1 u_3)\, (u_3-\alpha_2 u_1)}\, .
\end{split}
\ee
Across the locus $u_1=\alpha_3 u_2$,  the function $M_3(\alpha_i;u_i)$ is discontinuous,
and behaves as
\be
\label{M3jump}
M_3(\alpha_i;u_i) \sim - \sign(u_1-\alpha_3 u_2)\,
M_2\!\left(
\tfrac{\alpha_1(1+\alpha_3^2)+\alpha_2\alpha_3}{(1-\alpha_1\alpha_2\alpha_3)\sqrt{1+\alpha_3^2}};
\tfrac{\alpha_3 u_1+u_2-\alpha_1(1+\alpha_3^2)u_3}{\sqrt{(1+\alpha_3^2)[1+\alpha_1^2(1+\alpha_3^2)]}},
\tfrac{\alpha_1\alpha_3 u_1+\alpha_1 u_2+u_3}{\sqrt{1+\alpha_1^2(1+\alpha_3^2)}}
\right).
\ee
The behavior across across the loci $u_2=\alpha_1 u_3$ and $u_3=\alpha_2 u_1$
is obtained by circular permutations.
By acting with the operator $\tfrac{\I}{4}\, u_i\partial_{u_i}$ on \eqref{defE3} and
using \eqref{defE2int}, one finds that the shadow of $E_3$ is given by
\be
\begin{split}
\label{shadowE3}
\tfrac{\I}{4}\, u_i\partial_{u_i}\, E_3(\alpha_i;u_i) =&\,
\tfrac{\I}{2}\, \tfrac{u_1+\alpha_2(\alpha_1 u_2+u_3)}{\sqrt{1+(1+\alpha_1^2)\alpha_2^2}}\,
e^{-\frac{\pi[u_1+\alpha_2(\alpha_1 u_2+u_3)]^2}{1+(1+\alpha_1^2)\alpha_2^2}}\,
\\ & \times
E_2\!\left( \tfrac{\alpha_1}{\sqrt{1+(1+\alpha_1^2)\alpha_2^2}};
\tfrac{u_3-\alpha_2 u_1}{\sqrt{1+\alpha_2^2}},
\tfrac{(1+\alpha_2^2)u_2-\alpha_1\alpha_2(u_1+\alpha_2 u_3)}
{\sqrt{(1+\alpha_2^2)(1+(1+\alpha_1^2)\alpha_2^2)}}\right) + {\rm circ.}
\end{split}
\ee
The shadow of $M_3$ is similarly given by \eqref{shadowE3} with $E_2$ replaced by $M_2$.

More generally, for any triplet of vectors
$C_i$ spanning a time-like three-plane in $\IR^{3,n-3}$,
the functions $E_3(C_1,C_2,C_3;x)$ and $M_3(C_1,C_2,C_3;x)$ are equal to
$E_3(\alpha_i;u_i(x))$ and  $M_3(\alpha_i;u_i(x))$ where
$\alpha_i$ are determined from the conditions
\be
\tfrac{B(C_i^\star,C_j^\star)}{\sqrt{Q(C_i^\star) Q(C_j^\star)}} = -\tfrac{\alpha_k}{\sqrt{(1+\alpha_i^2)(1+\alpha_k^2)}}
\ee
for $(ijk)$ a circular permutation of $(123)$, and the linear forms $u_i(x)$ are given by
\be
u_i(x) = \frac{1}{1-\alpha_1\alpha_2\alpha_3}\left[
\sqrt{1+\alpha_j^2(1+\alpha_i^2)}\, \frac{(C_i,x)}{\sqrt{Q(C_i)}}  - \alpha_j
\sqrt{1+\alpha_i^2(1+\alpha_k^2)}\, \frac{(C_k,x)}{\sqrt{Q(C_k)}} \right].
\ee
Indeed, these definitions are invariant under $O(3,n-3)$, and in the special case where the $C_i$'s
are chosen as in \eqref{C123}, they reduce to \eqref{defE3}, \eqref{defM3}.

By construction, $E_3(C_1,C_2,C_3;x)$ is a $C^\infty$ solution  of Vign\'eras' equation
with $\lambda=0$ on $\IR^{3,n-3}$ which asymptotes to $\sign(C_1,x)\, \sign(C_2,x)\,\sign(C_3,x)$
as $|x|\to\infty$, while  $M_3(C_1,C_2,C_3;x)$ is a $C^\infty$ solution  of Vign\'eras' equation
with $\lambda=0$ on $\IR^{3,n-3}\backslash \{ B(C_1^\star,x) B(C_2^\star,x)\ B(C_3^\star,x)=0\}$,
exponentially \bp{decreasing}  as $|x|\to\infty$.
The shadow \eqref{shadowE3} of
$E_3$ becomes
\be
\begin{split}
\label{shadowE3C}
\tfrac{\I}{4}\, x\partial_x\, E_3(C_1,C_2,C_3;x) =&\,
\tfrac{\I}{2}\, \tfrac{B(C_1,x)}{\sqrt{Q(C_1)}}\,
e^{-\pi \frac{B(C_1,x)^2}{Q(C_1)}}\,
E_2\!\left(\Cn_{2\perp1},\Cn_{3\perp1}; x\right) + {\rm circ.}\, ,
\end{split}
\ee
while \eqref{M3jump} translates into
$M_3(C_1,C_2,C_3;x) \sim -\sign B(C_1^\star,x)\, M_2(C_2,C_3;x)$
near the locus $B(C_1^\star,x)=0$.
We claim that the following precise version of \eqref{EnMn} holds,
\be
\label{defE32}
\begin{split}
E_3&(C_1,C_2,C_3;x)= M_3(C_1,C_2,C_3;x)
+ \sign B(\Cdtp,x)\, M_2(C_2,C_3;x) +  \sign B(\Ctup,x)\, M_2(C_3,C_1;x)
\\
 &+ \sign B(\Cudp,x)\, M_2(C_1,C_2;x)
+  \sign B(\Cn_{2\perp1},x)\, \sign B(\Cn_{3\perp1},x)\, M_1(C_1;x)
\\
&+  \sign B(\Cn_{3\perp2},x)\, \sign B(\Cn_{1\perp2},x)\, M_1(C_2;x)
+ \sign B(\Cn_{1\perp3},x)\, \sign B(\Cn_{2\perp3},x)\, M_1(C_3;x)
\\
&+  \sign B(C_1,x)\, \sign B(C_2,x)\, \sign B(C_3,x) .
\end{split}
\ee
Indeed, the right-hand side of \eqref{defE32}
is  smooth across the loci $B(\Cdtp,x)=0$, $B(\Ctup,x)=0$, $B(\Cudp,x)=0$,
and has the same shadow and asymptotics at large $|u|$ as $E_3(C_1,C_2,C_3;x)$.
Thus, these two functions must coincide.
Moreover, in the asymptotic region where $|B(C_3,x)|\to\infty$ keeping  $B(C_1,x)$ and $B(C_2,x)$ finite,
one has $B(C_{3\perp1},x)\sim B(C_{3\perp2},x)\sim B(C_3^\star,x)\sim B(C_3,x)$, so  \eqref{defE32}
reduces to
\be
E_3(C_1,C_2,C_3;x) \sim \sign B(C_3,x) \, E_2(C_1,C_2;x) ,
\ee
in agreement with \eqref{Enlim}. Similarly, in the region where
both
$|B(C_2,x)|$, $|B(C_3,x)|\to\infty$ keeping $B(C_1,x)$ finite,
one has  $B(C_{2\perp1},x)\sim B(C_2,x)$, $B(C_{3\perp1},x)\sim B(C_3,x)$ so \eqref{defE32}
reduces to
\be
E_3(C_1,C_2,C_3;x) \sim \sign B(C_2,x) \, \sign B(C_3,x) \, E_1(C_1;x),
\ee
again in agreement with \eqref{Enlim}.

Following the same strategy as in Section \ref{sec_confth},
the triple error functions $E_3$ and $M_3$ can be used to find the modular completion of any indefinite
theta series of the form $\vartheta_{\bfmu}[\Phi_{3},0]$ with
$\Phi_3(x)=\tfrac18 \prod_{i=1,2,3} \bigl[\sign B(C_i,x)-\sign B(C'_i,x) \bigr]$. We leave the
detailed study of such theta series for future work.


\end{document}